\documentclass{amsart}

\usepackage{amsmath, amssymb, amsthm, bbm}

\newtheorem{theorem}{Theorem}[section]
\newtheorem{lemma}[theorem]{Lemma}
\newtheorem{proposition}[theorem]{Proposition}
\newtheorem{corollary}[theorem]{Corollary}
\newtheorem{claim}[theorem]{Claim}

\newtheorem{question}[theorem]{Question}

\newenvironment{definition}[1][Definition]{\begin{trivlist}
\item[\hskip \labelsep {\bfseries #1}]}{\end{trivlist}}

\newenvironment{remark}[1][Remark]{\begin{trivlist}
\item[\hskip \labelsep {\bfseries #1}]}{\end{trivlist}}

\newcommand{\cf}{\mathrm{cf}}

\newcommand{\bb}{\mathbb}
\newcommand{\bbm}{\mathbbm}

\begin{document}
\title{Covering properties and square principles}
\author{Chris Lambie-Hanson}
\address{Einstein Institute of Mathematics, Hebrew University of Jerusalem \\ Jerusalem, 91904, Israel}
\email{clambiehanson@math.huji.ac.il}
\thanks{This research was undertaken while the author was a Lady Davis Postdoctoral Fellow. The author would like to thank the Lady Davis Fellowship Trust and the Hebrew University of Jerusalem. The author would also like to thank Menachem Magidor for many helpful discussions and the anonymous referee for a number of useful comments.}
\begin{abstract}
	Covering matrices were used by Viale in his proof that the Singular Cardinals Hypothesis follows from the Proper Forcing Axiom and later by Sharon and Viale to investigate the impact of stationary reflection on the approachability ideal. In the course of this work, they isolated two reflection principles, $\mathrm{CP}$ and $\mathrm{S}$, which may hold of covering matrices. In this paper, we continue previous work of the author investigating connections between failures of $\mathrm{CP}$ and $\mathrm{S}$ and variations on Jensen's square principle. We prove that, for a regular cardinal $\lambda > \omega_1$, assuming large cardinals, $\square(\lambda, 2)$ is consistent with $\mathrm{CP}(\lambda, \theta)$ for all $\theta$ with $\theta^+ < \lambda$. We demonstrate how to force nice $\theta$-covering matrices for $\lambda$ which fail to satisfy $\mathrm{CP}$ and $\mathrm{S}$. We investigate normal covering matrices, showing that, for a regular uncountable $\kappa$, $\square_\kappa$ implies the existence of a normal $\omega$-covering matrix for $\kappa^+$ but that cardinal arithmetic imposes limits on the existence of a normal $\theta$-covering matrix for $\kappa^+$ when $\theta$ is uncountable. We introduce the notion of a good point for a covering matrix, in analogy with good points in PCF-theoretic scales. We develop the basic theory of these good points and use this to prove some non-existence results about covering matrices. Finally, we investigate certain increasing sequences of functions which arise from covering matrices and from PCF-theoretic considerations and show that a stationary reflection hypothesis places limits on the behavior of these sequences.
\end{abstract}
\maketitle

\section{Introduction}

The study of compactness and incompactness phenomena has long occupied a prominent place in set theory. Instances of incompactness (variations on Jensen's square principle, for example) typically abound in canonical inner models, while large cardinals or forcing axioms typically imply compactness or reflection principles. This sets up a tension between canonical inner models on one hand and large cardinals and forcing axioms on the other, and the investigation of this tension has proven to be quite fruitful in the study of canonical inner models, combinatorial set theory, and cardinal arithmetic (see, for example, \cite{cfm} or \cite{foremanmagidor}). In this paper, continuing in the spirit of \cite{lambiehanson}, we investigate connections between variations on Jensen's square principle and reflection principles related to covering matrices.

Covering matrices play an essential role in Viale's proof in \cite{viale} that the singular cardinals hypothesis (SCH) follows from the P-ideal dichotomy (PID) and also from the mapping reflection property (MRP), both of which are consequences of the proper forcing axiom (PFA). In \cite{vialeguessing}, again using covering matrices, Viale shows that SCH follows from the existence of sufficiently many internally unbounded $\aleph_1$-guessing models. In \cite{sharonviale}, Sharon and Viale use covering matrices to show that instances of approachability follow from hypotheses about stationary reflection. In these works, Sharon and Viale isolate reflection principles, $\mathrm{CP}$ and $\mathrm{S}$, which can hold of covering matrices. Certain instances of these principles follow from well-known set theoretic hypothesis, such as PFA or simultaneous stationary reflection, and in turn imply statements such as SCH or the failure of certain square principles. In \cite{lambiehanson}, the author considers covering matrices which form counterexamples to $\mathrm{CP}$ and $\mathrm{S}$ and their connections with square principles. We continue that work here. We also investigate the extent to which square principles may be compatible with instances of $\mathrm{CP}$ and $\mathrm{S}$ and derive some PCF-theoretic consequences from stationary reflection principles.

The outline of the paper is as follows. In Section \ref{cmsect}, we provide basic definitions and facts concerning covering matrices. In Section \ref{square2sect}, we show that, for a regular cardinal $\lambda$, the covering property $\mathrm{CP}^*(\lambda)$ is consistent with $\square(\lambda, 2)$. This shows that Viale's result from \cite{vialecovering} that $\mathrm{CP}(\lambda, \omega)$ implies the failure of $\square(\lambda)$ is sharp. In Section \ref{forcingsect}, we show how, given regular cardinals $\theta < \lambda$, one can force the existence of a closed, uniform, transitive $\theta$-covering matrix for $\lambda$, $\mathcal{D}$, such that $\mathrm{CP}(\mathcal{D})$ and $\mathrm{S}(\mathcal{D})$ both fail. In Section \ref{normalsect}, we obtain results on the possibility of the existence of normal $\theta$-covering matrices for $\lambda$ when $\theta < \lambda$ are regular cardinals, focusing in particular on the case $\theta = \omega$. We prove that, if $\kappa$ is a regular uncountable cardinal, then $\square_\kappa$ implies the existence of a normal $\omega$-covering matrix for $\kappa^+$. On the other hand, if $\theta < \kappa$ are cardinals with $\theta$ regular and uncountable, then the existence of a normal $\theta$-covering matrix for $\kappa^+$ implies $\kappa^\theta > \kappa$. In Section \ref{goodpointsect}, we develop the notion of a good point in a covering matrix. This is closely related to the notion of a good point in a PCF-theoretic scale. We develop the basic theory of good points and use this to prove non-existence results about certain covering matrices. In Section \ref{clubincreasesect}, we consider certain increasing sequence of functions which arise from PCF-theoretic considerations and show that stationary reflection places restrictions on their possible behavior.

Our notation is for the most part standard. The reference for all undefined notation and terminology is \cite{jech}. If $A$ is a set of ordinals, then $\mathrm{otp}(A)$ denotes the order type of $A$ and $A'$ denotes the set of limit points of $A$, i.e. the set of $\alpha \in A$ such that $\sup(A \cap \alpha) = \alpha$. If $\kappa$ is a cardinal, then $[A]^\kappa = \{X \subseteq A \mid |X| = \kappa \}$. $[A]^{\leq \kappa}$, $[A]^{<\kappa}$, etc. are defined in the obvious way. If $\theta < \lambda$ are cardinals, then $\lambda$ is \emph{$\theta$-inaccessible} if, for all $\kappa < \lambda$, $\kappa^\theta < \lambda$. The class of ordinals is denoted by $\mathrm{On}$. If $\lambda$ is an infinite cardinal and $f,g \in {^\lambda}\mathrm{On}$, then $f <^* g$ if $\{\alpha < \lambda \mid g(\alpha) \leq f(\alpha)\}$ is bounded in $\lambda$. If $\kappa < \lambda$ are cardinals and $\kappa$ is regular, then $S^\lambda_\kappa = \{\alpha < \lambda \mid \cf(\alpha) = \kappa\}$. $S^\lambda_{<\kappa}$, $S^\lambda_{\geq \kappa}$, etc. are defined in the obvious way.

\section{Covering matrices} \label{cmsect}

We first recall the definition of a covering matrix. Our definition matches that given by Sharon and Viale in \cite{sharonviale}. We note that similar notions existed prior to Viale's work. For example, in \cite{foremanmagidor}, Foreman and Magidor consider an object they call a Jensen matrix, which shares some of the basic properties of a covering matrix.

\begin{definition}
	Let $\theta<\lambda$ be regular cardinals. $\mathcal{D}= \langle D(i,\beta) \mid i<\theta, \beta<\lambda \rangle$ is a \emph{$\theta$-covering matrix for $\lambda$} if:
	\begin{enumerate}
		\item{for all $\beta<\lambda$, $\beta = \bigcup_{i<\theta}D(i,\beta)$;}
		\item{for all $\beta<\lambda$ and all $i<j<\theta$, $D(i,\beta) \subseteq D(j, \beta)$;}
		\item{for all $\beta<\gamma<\lambda$ and all $i<\theta$, there is $j<\theta$ such that $D(i,\beta)\subseteq D(j,\gamma)$.}
	\end{enumerate}

	$\beta_{\mathcal{D}}$ is the least $\beta$ such that for all $\gamma<\lambda$ and all $i<\theta$, $\mathrm{otp}(D(i,\gamma))<\beta$. $\mathcal{D}$ is \emph{normal} if $\beta_{\mathcal{D}}<\lambda$.

	$\mathcal{D}$ is \emph{transitive} if, for all $\alpha<\beta<\lambda$ and all $i<\theta$, if $\alpha \in D(i,\beta)$, then $D(i,\alpha)\subseteq D(i,\beta)$.

	$\mathcal{D}$ is \emph{uniform} if for all limit ordinals $\beta<\lambda$ there is $i<\theta$ such that $D(i,\beta)$ contains a club in $\beta$.

	$\mathcal{D}$ is \emph{closed} if for all $\beta<\lambda$, all $i<\theta$, and all $X\in [D(i,\beta)]^{\leq \theta}$, $\sup(X) \in D(i,\beta)$.

	$\mathcal{D}$ is \emph{downward coherent} if, for all $\alpha < \beta < \lambda$ and all $i < \theta$, there is $j < \theta$ such that $D(i,\beta) \cap \alpha \subseteq D(j,\alpha)$.

	$\mathcal{D}$ is \emph{locally downward coherent} if, for all $X \in [\lambda]^{\leq \theta}$, there is $\gamma_X < \lambda$ such that, for all $\beta < \lambda$ and all $i < \theta$, there is $j < \theta$ such that $X \cap D(i,\beta) \subseteq X \cap D(j, \gamma_X)$.

	$\mathcal{D}$ is \emph{strongly locally downward coherent} if, for all $X \in [\lambda]^{\leq \theta}$, there is $\gamma_X < \lambda$ such that, for all $\gamma_X \leq \beta < \lambda$, there is $i < \theta$ such that, for all $i \leq j < \theta$, $X \cap D(j,\beta) = X \cap D(j, \gamma_X)$. 
\end{definition}

We next introduce the reflection principles $\mathrm{CP}$ and $\mathrm{S}$, isolated by Sharon and Viale.

\begin{definition}
	Let $\theta < \lambda$ be regular cardinals, and let $\mathcal{D}$ be a $\theta$-covering matrix for $\lambda$. 
	\begin{enumerate}
		\item{$\mathrm{CP}(\mathcal{D}$) holds if there is an unbounded $T\subseteq \lambda$ such that for every $X \in [T]^\theta$, there are $i<\theta$ and $\beta < \lambda$ such that $X\subseteq D(i,\beta)$ (in this case, we say that $\mathcal{D}$ \emph{covers} $[T]^\theta$).}
		\item{$\mathrm{S}(\mathcal{D}$) holds if there is a stationary $S \subseteq \lambda$ such that, for every family $\{S_j \mid j<\theta\}$ of stationary subsets of $S$, there are $i<\theta$ and $\beta < \lambda$ such that, for every $j<\theta$, $S_j \cap D(i, \beta) \not= \emptyset$.}
		\item{If $\theta < \lambda$ are regular cardinals, then $\mathrm{CP}(\lambda, \theta)$ is the assertion that $\mathrm{CP}(\mathcal{D})$ holds whenever $\mathcal{D}$ is a locally downward coherent $\theta$-covering matrix for $\lambda$. $\mathrm{CP}^*(\lambda)$ is the assertion that $\mathrm{CP}(\lambda, \theta)$ holds for all regular $\theta$ such that $\theta^+ < \lambda$.}
	\end{enumerate}
\end{definition}

\begin{remark}
	$\mathrm{CP}$ stands for ``covering property." If $\lambda$ is a regular uncountable cardinal, then $\mathrm{CP}^*(\lambda)$ is the strongest one can hope for, since Viale shows in \cite{vialecovering} that $\mathrm{CP}(\kappa^+, \kappa)$ fails for every regular $\kappa$.
\end{remark}

The following is the basis of Viale's proof in \cite{viale} that SCH follows from PFA.

\begin{theorem}
	Suppose PFA holds. Then $\mathrm{CP}(\lambda, \omega)$ holds for all regular $\lambda > \omega_2$.
\end{theorem}

\begin{definition}
	Let $\theta < \lambda$ be regular cardinals. $\mathrm{R}(\lambda, \theta$) is the statement that there is a stationary $S\subseteq \lambda$ such that, for every family $\{S_j \mid j<\theta \}$ of stationary subsets of $S$, there is $\alpha < \lambda$ of uncountable cofinality such that, for all $j<\theta$, $S_j$ reflects at $\alpha$ (i.e. $S_j \cap \alpha$ is stationary in $\alpha$).
\end{definition}

If $\mathcal{D}$ is a nice enough covering matrix, then $\mathrm{CP}(\mathcal{D}$) and $\mathrm{S}(\mathcal{D}$) are equivalent and $\mathrm{R}(\lambda, \theta$) implies both. The following is proved in \cite{sharonviale}:

\begin{lemma} \label{implicationlemma}
	Let $\theta < \lambda$ be regular cardinals, and let $\mathcal{D}$ be a $\theta$-covering matrix for $\lambda$.
	\begin{enumerate}
 		\item{If $\mathcal{D}$ is transitive, then $\mathrm{S}(\mathcal{D}$) implies $\mathrm{CP}(\mathcal{D}$).}
 		\item{If $\mathcal{D}$ is closed, then $\mathrm{CP}(\mathcal{D}$) implies $\mathrm{S}(\mathcal{D}$).}
 		\item{If $\mathcal{D}$ is uniform, then $\mathrm{R}(\lambda, \theta$) implies $\mathrm{S}(\mathcal{D}$).}
	\end{enumerate}
\end{lemma}

Under some cardinal arithmetic, local downward coherence is not an additional assumption.

\begin{proposition} \label{downwardcoherent}
	Suppose $\theta < \lambda$ are regular cardinals, $2^\theta < \lambda$, and $\mathcal{D} = \langle D(i,\beta) \mid i < \theta, \beta < \lambda \rangle$ is a $\theta$-covering matrix for $\lambda$. Then the following hold:
	\begin{enumerate}
		\item{$\mathcal{D}$ is locally downward coherent;}
		\item{if $\mathcal{D}$ is transitive, then $\mathcal{D}$ is strongly locally downward coherent.}
	\end{enumerate}
\end{proposition}

\begin{proof}
	Fix $X \in [\lambda]^{\leq \theta}$. For each $\beta < \lambda$, define $g_\beta : \theta \rightarrow \mathcal{P}(X)$ by letting $g_\beta(i) = X \cap D(i,\beta)$ for all $i < \theta$. Since $2^\theta < \lambda$, there is a fixed $g : \theta \rightarrow \mathcal{P}(X)$ and an unbounded $T \subseteq \lambda$ such that, for all $\beta \in T$, $g_\beta = g$. Let $\gamma_X = \min(T)$. We claim that this works. To see this, fix $\beta < \lambda$ and $i < \theta$. We must produce a $j < \theta$ such that $X \cap D(i,\beta) \subseteq D(j,\gamma_X)$. Let $\beta' = \min(T \setminus \beta)$. Fix $j < \theta$ such that $D(i,\beta) \subseteq D(j,\beta')$. Then, since $\gamma_X, \beta' \in T$, $X \cap D(i,\beta) \subseteq X \cap D(j,\beta') = X \cap D(j, \gamma_X)$. This proves (1).

	To see (2), assume that $\mathcal{D}$ is transitive. Define $\langle g_\beta \mid \beta < \lambda \rangle$, $T$, and $\gamma_X$ as in the proof of (1). Fix $\beta$ with $\gamma_X < \beta < \lambda$. We must find $i < \theta$ such that for all $i \leq j < \theta$, $X \cap D(j,\beta) = X \cap D(j, \gamma_X)$. Let $\beta' = \min(T \setminus (\beta + 1))$, and let $i < \theta$ be such that $\beta \in D(i,\beta')$ and $\gamma_X \in D(i,\beta)$. Then, for all $i \leq j < \theta$, since $\mathcal{D}$ is transitive, we have $D(j,\gamma_X) \subseteq D(j, \beta) \subseteq D(j, \beta')$, so $X \cap D(j, \gamma_X) \subseteq X \cap D(j, \beta) \subseteq X \cap D(j, \beta') = X \cap D(j, \gamma_X)$, so we have equality.
\end{proof}

\section{Weak square and the covering property} \label{square2sect}

We recall here the definitions of certain square principles.

\begin{definition}
	Suppose $\lambda$ is a regular uncountable cardinal and $2 \leq \kappa \leq \lambda$. A $\square(\lambda, < \kappa)$-sequence is a sequence $\vec{\mathcal{C}} = \langle \mathcal{C}_\alpha \mid \alpha < \lambda \rangle$ such that:
	\begin{enumerate}
		\item{for all limit $\alpha < \lambda$, $1 \leq |\mathcal{C}_\alpha| < \kappa$;}
		\item{for all limit $\alpha < \lambda$ and $C \in \mathcal{C}_\alpha$, $C$ is club in $\alpha$;}
		\item{for all limit $\alpha < \beta < \lambda$ and all $C \in \mathcal{C}_\beta$, if $\alpha \in C'$, then $C \cap \alpha \in \mathcal{C}_\alpha$;}
		\item{there is no club $D \subseteq \lambda$ such that, for all $\alpha \in D'$, $D \cap \alpha \in \mathcal{C}_\alpha$.}
	\end{enumerate}
	$\square(\lambda, < \kappa)$ is the statement that there is a $\square(\lambda, < \kappa)$-sequence. $\square(\lambda, < \kappa^+)$ is typically written as $\square(\lambda, \kappa)$, and $\square(\lambda, 1)$ is written as $\square(\lambda)$.
\end{definition}

\begin{definition}
	Suppose $\mu$ is an infinite cardinal and $2 \leq \kappa \leq \mu^+$. A $\square_{\mu, < \kappa}$-sequence is a $\square(\mu^+, < \kappa)$ sequence $\vec{\mathcal{C}} = \langle \mathcal{C}_\alpha \mid \alpha < \mu^+ \rangle$ such that, for all $\alpha < \mu^+$ and all $C \in \mathcal{C}_\alpha$, $\mathrm{otp}(C) \leq \mu$. Again, $\square_{\mu, < \kappa}$ is the statement that there is a $\square_{\mu, < \kappa}$-sequence, $\square_{\mu, < \kappa^+}$ is written as $\square_{\mu, \kappa}$, and $\square_{\mu, 1}$ is written as $\square_\mu$.
\end{definition}

In \cite{vialecovering}, Viale shows that, if $\lambda$ is a regular uncountable cardinal, then $\mathrm{CP}(\lambda, \omega)$ implies the failure of $\square(\lambda)$. In this section, we show that this result is sharp. In particular, we prove that, assuming large cardinals, $\mathrm{CP}^*(\lambda)$ is consistent with $\square(\lambda, 2)$ and, if $\lambda = \kappa^+$, with $\square_{\kappa, 2}$. Note that PID and MRP, the principles from which $\mathrm{CP}(\lambda, \omega)$ is derived in \cite{viale}, both imply the failure of weaker square principles (see \cite{raghavan} and \cite{strullu}). 

We start with a technical lemma. In what follows, if $\kappa$ is a regular cardinal and $\bb{P}$ is a forcing poset, we say that $\bb{P}$ 
is \emph{$\kappa$-distributive} if every intersection of fewer than $\kappa$ dense open subsets of $\bb{P}$ is dense in $\bb{P}$. 
Equivalently, forcing with $\bb{P}$ adds no new sequences of length less than $\kappa$ of elements from the ground model.

\begin{lemma} \label{preservationlemma}
	Suppose $\theta < \lambda$ are infinite, regular cardinals and $\mathcal{D} = \langle D(i,\beta) \mid i < \theta, \beta < \lambda \rangle$ is a locally downward coherent $\theta$-covering matrix for $\lambda$. Suppose $\bb{Q}$ is a forcing poset such that $\bb{Q} \times \bb{Q}$ is $\theta^+$-distributive and, in $V^{\bb{Q}}$, there is a sequence of sets $\langle T_i \mid i < \theta \rangle$ satisfying the following:
	\begin{enumerate}
		\item{$\bigcup_{i<\theta} T_i = \lambda$;}
		\item{for all $i < j < \theta$, $T_i \subseteq T_j$;}
		\item{for all $i < \theta$ and $\beta < \lambda$, there is $j < \theta$ such that $D(i, \beta) \subseteq T_j$;}
		\item{for all $X \in [\lambda]^{\leq \theta}$ and all $i < \theta$, there is $j < \theta$ such that $X \cap T_i \subseteq D(j, \gamma_X)$.}
	\end{enumerate}
	Then $\mathrm{CP}(\mathcal{D})$ holds in $V$.
\end{lemma}

\begin{remark}
	The above assumptions about $\langle T_i \mid i < \theta \rangle$ essentially amount to saying that forcing with $\bb{Q}$ extends $\mathcal{D}$ by adding a column at $\lambda$.
\end{remark}

\begin{proof}
  Suppose for sake of contradiction that $\mathrm{CP}(\mathcal{D})$ fails in $V$. Let $\langle \dot{T}_i \mid i < \theta \rangle$ be a sequence of $\bb{Q}$-names forced to have the properties enumerated in the statement of the lemma. Since $\lambda$ is regular and $\bb{Q}$ is $\theta^+$-distributive, it is forced that, for large enough $i < \theta$, $\dot{T}_i$ is unbounded in $\lambda$. Thus, by forcing below a stronger condition and moving to a subsequence if necessary, we may assume without loss of generality that $\dot{T}_0$ is forced to be unbounded in $\lambda$. By condition (4) in the statement of the lemma, we have that, for all $i < \theta$, $\Vdash_{\bb{Q}} ``[\dot{T}_i]^{\leq \theta}$ is covered by $\mathcal{D}."$ In particular, since $\mathrm{CP}(\mathcal{D})$ fails in $V$, we have $\Vdash_{\bb{Q}} ``\dot{T}_i \not\in V."$

	\begin{claim} \label{differenceclaim}
		Suppose $i,j < \theta$. Then the set $E_{i,j} = \{(q_0, q_1) \in \bb{Q} \times \bb{Q} \mid$ for some $\alpha < \lambda$, $q_0 \Vdash_{\bb{Q}} ``\alpha \in \dot{T}_i"$ and $q_1 \Vdash_{\bb{Q}} ``\alpha \not\in \dot{T}_j"\}$ is dense in $\bb{Q} \times \bb{Q}$. 
	\end{claim}

	\begin{proof}
		Fix $i,j < \theta$ and $(p_0, p_1) \in \bb{Q} \times \bb{Q}$. We will find $(q_0, q_1) \leq (p_0, p_1)$ with $(q_0, q_1) \in E_{i,j}$.

		Let $A_0 = \{\alpha < \lambda \mid$ for some $q \leq p_0$, $q \Vdash ``\alpha \in \dot{T}_i"\}$, and let $A_1 = \{\alpha < \lambda \mid$ for some $q \leq p_1$, $q \Vdash ``\alpha \in \dot{T}_j"\}$. Note that, since $\dot{T}_i$ and $\dot{T}_j$ are forced to be unbounded in $\lambda$, $A_0$ and $A_1$ are themselves unbounded in $\lambda$. Since $A_0$ is an unbounded subset of $\lambda$ in $V$ and $\mathrm{CP}(\mathcal{D})$ fails, there is $X \in [A_0]^{\leq \theta}$ such that, for all $k < \theta$ and $\beta < \lambda$, $X \not\subseteq D(k,\beta)$. By condition (4) in the statement of the lemma, this implies that $\Vdash ``X \not\subseteq \dot{T}_j."$ We can therefore find $q_1 \leq p_1$ and $\alpha \in X$ such that $q_1 \Vdash ``\alpha \not\in \dot{T}_j."$ Since $\alpha \in A_0$, we can find $q_0 \leq p_0$ such that $q_0 \Vdash ``\alpha \in \dot{T}_i."$ Then $(q_0, q_1) \leq (p_0, p_1)$ and $(q_0, q_1) \in E_{i,j}$. 
	\end{proof}
	Let $G_0 \times G_1$ be $\bb{Q} \times \bb{Q}$-generic over $V$, and move to $V[G_0 \times G_1]$. For $\epsilon < 2$, let $\langle T^\epsilon_i \mid i < \theta \rangle$ be the realization of $\langle \dot{T}_i \mid i < \theta \rangle$ in $V[G_\epsilon]$. By Claim \ref{differenceclaim}, for every $j < \theta$, we can fix $\alpha_j \in T^0_0 \setminus T^1_j$. Let $X = \{\alpha_j \mid j < \theta\}$. Since $\bb{Q} \times \bb{Q}$ is $\theta^+$-distributive, we have $X \in V$. By condition (4) in the statement of the lemma, we can fix $i < \theta$ such that $X = X \cap T^0_0 \subseteq D(i, \gamma_X)$. By condition (3), we can find $j < \theta$ such that $D(i, \gamma_X) \subseteq T^1_j$. But then $X \subseteq T^1_j$, so, in particular, $\alpha_j \in T^1_j$, which is a contradiction.
\end{proof}

For our forcing arguments, we will need the notion of \emph{strategic closure}.

\begin{definition}
	Let $\mathbb{P}$ be a partial order and let $\beta$ be an ordinal.
	\begin{enumerate}
		\item {The two-player game $G_\beta(\mathbb{P})$ is defined as follows: Players I and II alternately play entries in $\langle p_\alpha \mid \alpha < \beta \rangle$, a decreasing sequence of conditions in $\mathbb{P}$ with $p_0 = \bbm{1}_{\mathbb{P}}$. Player I plays at odd stages, and Player II plays at even stages (including all limit stages). If there is an even stage $\alpha < \beta$ at which Player II can not play, then Player I wins. Otherwise, Player II wins.}
		\item{$\mathbb{P}$ is said to be {\em $\beta$-strategically closed} if Player II has a winning strategy for the game $G_\beta(\mathbb{P})$.}
	\end{enumerate}
\end{definition}

\begin{remark}
  By an easy argument, if $\kappa$ is a cardinal, $\bb{P}$ is a forcing poset, and $\bb{P}$ is $(\kappa + 1)$-strategically closed, then 
  $\bb{P}$ is $\kappa^+$-distributive.
\end{remark}

We now introduce some relevant forcing posets. Let $\lambda$ be a regular uncountable cardinal. We first define a poset $\bb{S}(\lambda)$ designed to add a $\square(\lambda, 2)$-sequence. Conditions of $\bb{S}(\lambda)$ are of the form $s = \langle \mathcal{C}^s_\alpha \mid \alpha \leq \gamma^s \rangle$, where:

\begin{itemize}
	\item{$\gamma^s < \lambda$;}
	\item{for all limit ordinals $\alpha \leq \gamma^s$, $1 \leq |\mathcal{C}_\alpha| \leq 2$;}
	\item{for all limit $\alpha \leq \gamma^s$ and all $C \in \mathcal{C}_\alpha$, $C$ is a club in $\alpha$;}
	\item{for all limit $\alpha < \beta \leq \gamma^s$ and all $C \in \mathcal{C}_\beta$, if $\alpha \in C'$, then $C \cap \alpha \in \mathcal{C}_\alpha$.}
\end{itemize}

If $s_0,s_1 \in \bb{S}(\lambda)$, then $s_1 \leq s_0$ if and only if $s_1$ end-extends $s_0$, i.e. $\gamma^{s_1} \geq \gamma^{s_0}$ and, for all $\alpha \leq \gamma^{s_0}$, $\mathcal{C}^{s_1}_\alpha = \mathcal{C}^{s_0}_\alpha$.

The following facts are standard. See, for example, \cite{lambiehanson} for proofs. (In that paper, we are considering the forcing to add a $\square(\lambda)$-sequence, but the same proofs work in this case.)

\begin{proposition}
	Let $\lambda$ be a regular uncountable cardinal, and let $\bb{S} = \bb{S}(\lambda)$.
	\begin{enumerate}
		\item{$\bb{S}$ is countably closed.}
		\item{$\bb{S}$ is $\lambda$-strategically closed.}
		\item{If $G$ is $\bb{S}$-generic over $V$, then $\bigcup G = \langle \mathcal{C}_\alpha \mid \alpha < \lambda \rangle$ is a $\square(\lambda, 2)$-sequence.}
	\end{enumerate}
\end{proposition}

Next, we introduce a forcing poset designed to thread a $\square(\lambda, 2)$ sequence. Let $\vec{\mathcal{C}} = \langle \mathcal{C}_\alpha \mid \alpha < \lambda \rangle$ be a $\square(\lambda, 2)$-sequence. $\bb{T}(\vec{\mathcal{C}})$ is the forcing poset whose conditions are closed, bounded subsets $t$ of $\lambda$ such that, for every $\alpha \in t'$, $t \cap \alpha \in \mathcal{C}_\alpha$. If $t_0, t_1 \in \bb{T}(\vec{\mathcal{C}})$, then $t_1 \leq t_0$ if and only if $t_1$ end-extends $t_0$.

\begin{proposition} \label{denseclosed}
	Let $\lambda$ be a regular uncountable cardinal, and let $\bb{S} = \bb{S}(\lambda)$. Let $\dot{\vec{\mathcal{C}}}$ be the canonical $\bb{S}$-name for the $\square(\lambda, 2)$-sequence added by the generic filter, and let $\dot{\bb{T}}$ be an $\bb{S}$-name for $\bb{T}(\dot{\vec{\mathcal{C}}})$. Then $\bb{S} * (\dot{\bb{T}} \times \dot{\bb{T}})$ has a dense $\lambda$-closed subset.
\end{proposition}

\begin{proof}
	Let $\bb{U}$ be the set of $(s, (\dot{t}_0, \dot{t}_1)) \in \bb{S} * (\dot{\bb{T}} \times \dot{\bb{T}})$ such that:
	\begin{itemize}
		\item{there are $t_0, t_1 \in V$ such that $s \Vdash_{\bb{S}} ``\dot{t}_0 = t_0$ and $\dot{t}_1 = t_1"$;}
		\item{$\gamma^s = \max(t_0) = \max(t_1)$.}
	\end{itemize}

	We claim that $\bb{U}$ is the desired dense $\lambda$-closed subset. We first show that it is dense. To this end, let $(s', (\dot{t}'_0, \dot{t}'_1)) \in \bb{S} * (\dot{\bb{T}} \times \dot{\bb{T}})$. First, find $s \leq s'$ and $t'_0, t'_1 \in V$ such that $s \Vdash_{\bb{S}} ``\dot{t}'_0 = t'_0$ and $\dot{t}'_1 = t'_1."$ This is possible, since $\bb{S}$ is $\lambda$-strategically closed and hence $\lambda$-distributive. By extending $s$ further if necessary, we may assume that $\gamma^{s} > \max(t'_0), \max(t'_1)$. Now let $t_0 = t'_0 \cup \{\gamma^s\}$ and $t_1 = t'_1 \cup \{\gamma^s\}$. It is easily verified that $(s, (t_0, t_1)) \leq (s', (\dot{t}'_0, \dot{t}'_1))$ and $(s, (t_0, t_1)) \in \bb{U}$.

	We now show that $\bb{U}$ is $\lambda$-closed. Let $\delta < \lambda$ be a limit ordinal, and suppose $\langle (s_\eta, (\dot{t}_{0, \eta}, \dot{t}_{1, \eta})) \mid \eta < \delta \rangle$ is a strictly decreasing sequence from $\bb{U}$. Since, for each $\eta < \delta$, $s_\eta$ decides the value of $\dot{t}_{0, \eta}$ and $\dot{t}_{1, \eta}$, we omit the dots and assume they are actual closed, bounded subsets of $\lambda$ in $V$ rather than $\bb{S}$-names. Let $\gamma = \sup(\{\gamma^{s_\eta} \mid \eta < \delta\})$. Let $t'_0 = \bigcup_{\eta < \delta} t_{0, \eta}$ and $t'_1 = \bigcup_{\eta < \delta} t_{1, \eta}$. Note that $\gamma = \sup(t'_0) = \sup(t'_1)$. We now define $(s, (t_0, t_1)) \in \bb{U}$ that will be a lower bound for the sequence. $s$ will be of the form $\langle \mathcal{C}^s_\alpha \mid \alpha \leq \gamma \rangle$. If $\alpha < \gamma$, then let $\mathcal{C}^s_\alpha = \mathcal{C}^{s_\eta}_\alpha$ for some $\eta < \delta$ such that $\alpha \leq \gamma^{s_\eta}$. Let $\mathcal{C}^s_\gamma = \{t'_0, t'_1\}$. Finally, let $t_0 = t'_0 \cup \{\gamma\}$, and $t_1 = t'_1 \cup \{\gamma\}$. It is easy to check that $(s, (t_0, t_1)) \in \bb{U}$ and is a lower bound for $\langle (s_\eta, (t_{0,\eta}, t_{1, \eta})) \mid \eta < \delta \rangle$.
\end{proof}

\begin{remark}
	Due to the fact that our forcings are tree-like (i.e. if $q \leq p_0, p_1$, then $p_0$ and $p_1$ are comparable), the subset $\bb{U}$ isolated above is actually $\lambda$-directed closed. Also, Proposition \ref{denseclosed} easily implies that $\bb{S} * \dot{\bb{T}}$ also has a dense $\lambda$-directed closed subset.
\end{remark}

The proof of Proposition \ref{denseclosed} easily extends to yield the following.

\begin{proposition} \label{denseclosed2}
	Let $\lambda$, $\bb{S}$, and $\dot{\bb{T}}$ be as in the statement of Proposition \ref{denseclosed}. Suppose that $\dot{\bb{R}}$ is an $\bb{S} * \dot{\bb{T}}$-name for a $\lambda$-closed forcing poset. Then $\bb{S} * ((\dot{\bb{T}} * \dot{\bb{R}}) \times (\dot{\bb{T}} * \dot{\bb{R}}))$ has a dense $\lambda$-closed subset.
\end{proposition}

We now prove the main result of this section in three distinct cases. We first look at the case of inaccessible cardinals.

\begin{theorem}
	Suppose $\lambda$ is a supercompact cardinal indestructible under $\lambda$-directed closed forcing. Let $\bb{S} = \bb{S}(\lambda)$, and let $G$ be $\bb{S}$-generic over $V$. Then, in $V[G]$, $\lambda$ is inaccessible and $\square(\lambda, 2)$ and $\mathrm{CP}^*(\lambda)$ both hold.
\end{theorem}

\begin{proof}
	The fact that $\lambda$ is inaccessible and $\square(\lambda, 2)$ holds in $V[G]$ is immediate. We just need to verify that $\mathrm{CP}^*(\lambda)$ holds. Thus, fix a regular cardinal $\theta < \lambda$, and let $\mathcal{D} = \langle D(i, \beta) \mid i < \theta, \beta < \lambda \rangle$ be a $\theta$-covering matrix for $\lambda$ in $V[G]$. We will show that $\mathrm{CP}(\mathcal{D})$ holds. Note that, since $\lambda$ is inaccessible, Proposition \ref{downwardcoherent} implies that $\mathcal{D}$ is locally downward coherent, so $\gamma_X$ is defined for all $X \in [\lambda]^{\leq \theta}$. Let $\vec{\mathcal{C}}$ be the $\square(\lambda, 2)$-sequence introduced by $G$, and let $\bb{T} = \bb{T}(\vec{\mathcal{C}})$. Let $\dot{\bb{T}} \in V$ be an $\bb{S}$-name for $\bb{T}$. Let $H$ be $\bb{T}$-generic over $V[G]$. Since $\bb{S} * \dot{\bb{T}}$ has a dense $\lambda$-directed closed subset in $V$ and $\lambda$ was indestructibly supercompact in $V$, $\lambda$ is supercompact in $V[G*H]$. Let $j : V[G*H] \rightarrow M$ be an elementary embedding with $\mathrm{crit}(j) = \lambda$.

	In $M$, $j(\mathcal{D})$ is a locally downward coherent $\theta$-covering matrix for $j(\lambda)$. For $i < \theta$ and $\beta < \lambda$, $j(D(i,\beta)) = D(i, \beta)$. Thus, we can write $j(\mathcal{D}) = \langle D(i, \beta) \mid i < \theta, \beta < j(\lambda) \rangle$. Moreover, for all $X \in [\lambda]^{\leq \theta}$, we have $(\gamma_X)^\mathcal{D} = (\gamma_X)^{j(\mathcal{D})}$. Therefore, $\langle D(i, \lambda) \mid i < \theta \rangle$ satisfies the conditions placed on $\langle T_i \mid i < \theta \rangle$ in the statement of Lemma \ref{preservationlemma}, and $\langle D(i, \lambda) \mid i < \theta \rangle \in V[G*H]$. Moreover, Proposition \ref{denseclosed} implies that, in $V[G]$, $\bb{T} \times \bb{T}$ is $\lambda$-distributive (and hence $\theta^+$-distributive). Therefore, we may apply Proposition \ref{preservationlemma} to conclude that $\mathrm{CP}(\mathcal{D})$ holds in $V[G]$.
\end{proof}

In the successor case, when $\lambda = \kappa^+$, we can actually arrange for $\square_{\kappa, 2}$ to hold. To do this, we slightly modify our posets $\bb{S}$ and $\bb{T}$. First, given an uncountable cardinal $\kappa$, let $\bb{S}^*(\kappa)$ be the forcing poset whose conditions are of the form $s = \langle \mathcal{C}_\alpha \mid \alpha \leq \gamma^s \rangle$ such that:
\begin{itemize}
	\item{$\gamma^s < \kappa^+$;}
	\item{for all limit ordinals $\alpha \leq \gamma^s$, $1 \leq |\mathcal{C}_\alpha| \leq 2$;}
	\item{for all limit $\alpha \leq \gamma^s$ and all $C \in \mathcal{C}_\alpha$, $C$ is a club in $\alpha$ and $\mathrm{otp}(C) \leq \kappa$;}
	\item{for all limit $\alpha < \beta \leq \gamma^s$ and all $C \in \mathcal{C}_\beta$, if $\alpha \in C'$, then $C \cap \alpha \in \mathcal{C}_\alpha$.}
\end{itemize}
If $s_0, s_1 \in \bb{S}^*(\kappa)$, then $s_1 \leq s_0$ if and only if $s_1$ end-extends $s_0$. A straightforward adaptation of the analogous proofs for the poset $\bb{S}(\lambda)$ yields the following.

\begin{proposition}
	Let $\kappa$ be an uncountable cardinal, and let $\bb{S} = \bb{S}^*(\kappa)$.
	\begin{enumerate}
		\item{$\bb{S}$ is countably closed.}
		\item{$\bb{S}$ is $(\kappa + 1)$-strategically closed.}
		\item{If $G$ is $\bb{S}$-generic over $V$, then $\bigcup G = \langle \mathcal{C}_\alpha \mid \alpha < \kappa^+ \rangle$ is a $\square_{\kappa, 2}$-sequence.}
	\end{enumerate}
\end{proposition}

If $\vec{\mathcal{C}} = \langle \mathcal{C}_\alpha \mid \alpha < \kappa^+ \rangle$ is a $\square_{\kappa, 2}$-sequence and $\nu \leq \kappa$ is a regular cardinal, let $\bb{T}^*_\nu(\vec{\mathcal{C}})$ be the poset whose conditions are $t$ such that:
\begin{itemize}
	\item{$t$ is a closed, bounded subset of $\kappa^+$;}
	\item{$\mathrm{otp}(t) < \nu$;}
	\item{for all $\alpha \in t'$, $t \cap \alpha \in \mathcal{C}_\alpha$.}
\end{itemize}
If $t_0, t_1 \in \bb{T}^*_\nu(\vec{\mathcal{C}})$, then $t_1 \leq t_0$ if and only if $t_1$ end-extends $t_0$. The proof of the following is similar to that of Proposition \ref{denseclosed}.

\begin{proposition} \label{denseclosed3}
	Let $\nu \leq \kappa$ be uncountable cardinals, with $\nu$ regular, and let $\bb{S} = \bb{S}^*(\kappa)$. Let $\dot{\vec{\mathcal{C}}}$ be the canonical $\bb{S}$-name for the $\square_{\kappa, 2}$-sequence added by the generic object, and let $\dot{\bb{T}}$ be an $\bb{S}$-name for $\bb{T}^*_\nu(\dot{\vec{\mathcal{C}}})$. Then $\bb{S} * (\dot{\bb{T}} \times \dot{\bb{T}})$ has a dense $\nu$-closed subset. Moreover, if $\dot{\bb{R}}$ is an $\bb{S} * \dot{\bb{T}}$-name for a $\nu$-closed forcing, then $\bb{S} * ((\dot{\bb{T}} * \dot{\bb{R}}) \times (\dot{\bb{T}} * \dot{\bb{R}}))$ has a dense $\nu$-closed subset. 
\end{proposition}

We now consider successors of regular cardinals.

\begin{theorem}
	Assume GCH. Suppose $\kappa < \lambda$ are regular uncountable cardinals, with $\lambda$ measurable. Let $\bb{P} = \mathrm{Coll}(\kappa, < \lambda)$, and let $\dot{\bb{S}}$ be a $\bb{P}$-name for $\bb{S}^*(\kappa)$. Then, in $V^{\bb{P} * \dot{\bb{S}}}$, $\lambda = \kappa^+$ and $\square_{\kappa, 2}$ and $\mathrm{CP}^*(\lambda)$ both hold.
\end{theorem}

\begin{proof}
	Let $G$ be $\bb{S}$-generic over $V$, and let $H$ be $\bb{T}$-generic over $V[G]$. Again, it is immediate that $\lambda = \kappa^+$ and $\square_{\kappa, 2}$ holds in $V[G*H]$, so we need only prove $\mathrm{CP}^*(\lambda)$. Thus, working in $V[G*H]$, let $\theta < \kappa$ be a regular cardinal, and let $\mathcal{D} = \langle D(i, \beta) \mid i < \theta, \beta < \lambda \rangle$ be a $\theta$-covering matrix for $\lambda$. Since $\lambda$ was inaccessible in $V$ and $\bb{P} * \dot{\bb{S}}$ is $\kappa$-distributive, we have $2^\theta < \lambda$ in $V[G*H]$, so Proposition \ref{downwardcoherent} implies that $\mathcal{D}$ is locally downward coherent. Let $\vec{\mathcal{C}} = \langle \mathcal{C}_\alpha \mid \alpha < \lambda \rangle$ be the $\square_{\kappa, 2}$-sequence added by $H$, and let $\bb{T} = \bb{T}^*_\kappa(\vec{\mathcal{C}})$. Let $\dot{\bb{T}}$ be a $\bb{P} * \dot{\bb{S}}$-name for $\bb{T}$.

  Moving back to $V$, let $j:V \rightarrow M$ be an elementary embedding with $\mathrm{crit}(j) = \lambda$. $j(\bb{P}) = \mathrm{Coll}(\kappa, < j(\lambda))$. Since $\bb{P} * \dot{\bb{S}} * \dot{\bb{T}}$ has a $\kappa$-closed dense subset and has size less than $j(\lambda)$, a result of Magidor from \cite{magidor} implies that $j(\bb{P}) \cong \bb{P} * \dot{\bb{S}} * \dot{\bb{T}} * \dot{\bb{R}}$, where $\dot{\bb{R}}$ is a name for a forcing poset that is forced to be $\kappa$-closed. Thus, letting $I$ be $\bb{T}$-generic over $V[G*H]$ and $J$ be $\bb{R}$-generic over $V[G*H*I]$, we can extend $j$ to $j:V[G] \rightarrow M[G*H*I*J]$. 

	In $M[G*H*I*J]$, let $E = \bigcup I$. Note that $E$ is a thread through $\vec{\mathcal{C}}$ of order type $\kappa$. Also, if we define $q = \langle \mathcal{C}^q_\alpha \mid \alpha \leq \lambda \rangle$ by letting $\mathcal{C}^q_\alpha = \mathcal{C}_\alpha$ for $\alpha < \lambda$ and $\mathcal{C}^q_\lambda = \{E\}$, we have $q \in j(\bb{S})$ and $q \leq j(s) = s$ for all $s \in H$. Thus, if we let $H^+$ be $j(\bb{S})$-generic over $V[G*H*I*J]$ with $q \in H^+$, we may extend $j$ to $j:V[G*H] \rightarrow M[G*H*I*J*H^+]$. 

	In $M[G*H*I*J*H^+]$, $j(\mathcal{D})$ is a locally downward coherent $\theta$-covering matrix for $j(\lambda)$. Moreover, $\gamma_X^{\mathcal{D}} = \gamma_X^{j(\mathcal{D})}$ for all $X \in [\lambda]^{\leq \theta}$. Thus, $\langle D(i, \lambda) \mid i < \theta \rangle$ satisfies the conditions in the statement of Lemma \ref{preservationlemma}. Since $j(\bb{S})$ is $j(\lambda)$-distributive in $V[G*H*I*J]$, $\langle D(i, \lambda) \mid i < \theta \rangle \in V[G*H*I*J]$ and satisfies the same conditions there. Moreover, by Proposition \ref{denseclosed3}, $(\bb{T} * \dot{\bb{R}}) \times (\bb{T} * \dot{\bb{R}})$ is $\kappa$-distributive in $V[G*H]$, so, since $\theta^+ \leq \kappa$, we can apply Lemma \ref{preservationlemma} to conclude that $\mathrm{CP}(\mathcal{D})$ holds in $V[G*H]$.
\end{proof}

We finally address successors of singular cardinals. For concreteness, we concentrate on $\aleph_{\omega + 1}$, but similar techniques will work at other cardinals.

\begin{theorem}
	Suppose there are infinitely many supercompact cardinals. Then there is a forcing extension in which $\square_{\aleph_\omega, 2}$ and $\mathrm{CP}^*(\aleph_{\omega + 1})$ both hold.
\end{theorem}

\begin{proof}
	Let $\langle \kappa_n \mid n < \omega \rangle$ be an increasing sequence of cardinals, with $\kappa_0 = \omega$ and $\kappa_n$ supercompact for all $0 < n < \omega$. Let $\mu = \sup(\{\kappa_n \mid n < \omega\})$, and let $\lambda = \mu^+$. Define a full support forcing iteration $\langle \bb{P}_m, \dot{\bb{Q}}_n \mid m \leq \omega, n < \omega \rangle$ by letting, for all $n < \omega$, $\dot{\bb{Q}}_n$ be a $\bb{P}_n$ name for $\mathrm{Coll}(\kappa_n, < \kappa_{n+1})$. Let $\bb{P} = \bb{P}_\omega$. For all $n < \omega$, let $\dot{\bb{P}}^n$ be a $\bb{P}_n$-name such that $\bb{P} \cong \bb{P}_n * \dot{\bb{P}}^n$. Let $G$ be $\bb{P}$-generic over $V$. For $n < \omega$, let $G_n$ and $G^n$ be the generic filters induced by $G$ on $\bb{P}_n$ and $\bb{P}^n$, respectively, so $V[G] = V[G_n * G^n]$. Standard arguments show that, in $V[G]$, $\kappa_n = \aleph_n$ for all $n < \omega$, $\mu = \aleph_\omega$, and $\lambda = \aleph_{\omega + 1}$.

	In $V[G]$, let $\bb{S} = \bb{S}^*(\mu)$. Let $H$ be $\bb{S}$-generic over $V[G]$. We claim that $V[G*H]$ is our desired model. It is clear that $\square_{\aleph_\omega, 2}$ holds in $V[G]$. Thus, let $\theta < \lambda$ be a regular cardinal, and let $\mathcal{D} = \langle D(i, \beta) \mid i < \theta, \beta < \lambda \rangle$ be a $\theta$-covering matrix for $\lambda$ in $V[G*H]$. $\mu$ is strong limit in $V[G*H]$, so Proposition \ref{downwardcoherent} again implies that $\mathcal{D}$ is locally downward coherent.

  Find $n^* < \omega$ such that $\theta < \kappa_{n^*}$. Let $\kappa^* = \kappa_{n^*+1}$. In $V[G*H]$, let $\vec{\mathcal{C}} = \langle \mathcal{C}_\alpha \mid \alpha < \lambda \rangle$ be the $\square_{\mu, 2}$-sequence added by $H$. Let $\bb{T} = \bb{T}^*_{\kappa^*}(\vec{\mathcal{C}})$, and let $\dot{\bb{T}}$ be a $\bb{P} * \dot{\bb{S}}$-name for $\bb{T}$. In $V[G_{n^*}]$, $\kappa^*$ is supercompact, so we can fix an elementary embedding $j:V[G_{n^*}] \rightarrow M[G_{n^*}]$ such that $\mathrm{crit}(j) = \kappa^*$, $j(\kappa^*) > \lambda$, and ${^\lambda}M[G_{n^*}] \subseteq M[G_{n^*}]$. $j(\bb{Q}_{n^*}) = \mathrm{Coll}(\kappa_{n^*}, < j(\kappa^*))$. Since $\bb{P}^{n^*} * \dot{\bb{S}} * \dot{\bb{T}}$ has a dense, $\kappa^*$-closed subset, we have, again by Magidor's result in \cite{magidor}, that $j(\bb{Q}_{n^*}) \cong \bb{P}^{n^*} * \dot{\bb{S}} * \dot{\bb{T}} * \dot{\bb{R}}$, where $\dot{\bb{R}}$ is forced to be $\kappa_{n^*}$-closed. Thus, letting $I$ be $\bb{T}$-generic over $V[G*H]$ and $J$ be $\bb{R}$-generic over $V[G*H*I]$, we can extend $j$ to $j:V[G_{n^*+1}] \rightarrow M[G*H*I*J]$.

	In $M[G*H*I*J]$, $j(\bb{P}^{n^*+1})$ is $j(\kappa^*)$-directed closed. By standard arguments, $M[G*H*I*J]$ is closed under $\lambda$-sequences. Therefore, $j``G^{n^*+1} \in M[G*H*I*J]$ is a directed subset of $j(\bb{P}^{n^*+1})$ of size $\lambda$, so it has a lower bound, $q$, in $j(\bb{P}^{n^*+1})$. Thus, if we let $G^+$ be $j(\bb{P}^{n^*+1})$-generic over $V[G*H*I*J]$ with $q \in G^+$, we may lift $j$ further to $j:V[G] \rightarrow M[G*H*I*J*G^+]$. Let $\eta = \sup(j``\lambda)$. Letting $E = \bigcup_{t \in I} j(t)$, we see that $E$ is a club in $\eta$ of order type $\kappa^*$. We can then define $q \in j(\bb{S})$, $q = \langle \mathcal{C}^q_\alpha \mid \alpha \leq \eta \rangle$ as follows. For $\alpha < \eta$, find $s \in H$ such that $j(\gamma^s) \geq \alpha$, and let $\mathcal{C}^q_\alpha = \mathcal{C}^{j(s)}_\alpha$. Then, let $\mathcal{C}^q_\eta = \{E\}$. It is easily verified that $q \in j(\bb{S})$ and $q \leq j(s)$ for all $s \in H$. Thus, if we let $H^+$ be $j(\bb{S})$-generic over $V[G*H*I*J*G^+]$ with $q \in H^+$, we can lift $j$ one last time to $j:V[G*H] \rightarrow M[G*H*I*J*G^+ * H^+]$.

  $j(\mathcal{D}) = \langle D^*(i, \beta) \mid i < \theta, \beta < j(\lambda)\rangle$ is a locally downward coherent $\theta$-covering matrix for $j(\lambda)$. For $i < \theta$, let $T_i = \{\alpha < \lambda \mid j(\alpha) \in D^*(i, \eta)\}$. It is routine to verify that $\langle T_i \mid i < \theta \rangle$ satisfies conditions (1)-(3) in the statement of Lemma \ref{preservationlemma} with respect to $\mathcal{D}$. We verify condition (4). To this end, fix $X \in [\lambda]^{\leq \theta}$. Since, in $V[G*H]$, $\bb{T} * \dot{\bb{R}} * j(\bb{P}^{n^*+1}) * j(\bb{S})$ is $\theta^+$-distributive, we have $X \in V[G*H]$. $\mathcal{D}$ is a locally downward coherent covering matrix in $V[G*H]$, so $\gamma_X^\mathcal{D} < \lambda$ exists. Since the critical point of $j$ is greater than $\theta$, we have $j(X) = j``X$, so $\gamma_{j(X)}^{j(\mathcal{D})} = \gamma_{j``X}^{j(\mathcal{D})} = j(\gamma_X^\mathcal{D}) < \eta$. Therefore, for all $i < \theta$, there is $k < \theta$ such that $j``X \cap D^*(i, \eta) \subseteq j``X \cap D^*(k, j(\gamma_X^\mathcal{D}))$. But, for all $i,k < \theta$, $X \cap T_i = j^{-1}[j``X \cap D^*(i, \eta)]$ and $X \cap D(k, \gamma_X^\mathcal{D}) = j^{-1}[j``X \cap D^*(k, j(\gamma_X^\mathcal{D}))]$. Therefore, for all $i < \theta$, there is $k < \theta$ such that $X \cap T_i \subseteq D(k, \gamma_X^\mathcal{D})$, as required.
  
 Since $j(\bb{P}^{n^*+1}) * j(\bb{S})$ is $j(\kappa^*)$-distributive in $M[G*H*I*J]$ and $M[G*H*I*J]$ is closed under $\lambda$-sequences, we actually have $\langle T_i \mid i < \theta \rangle \in V[G*H*I*J]$. By Proposition \ref{denseclosed3}, we also know that, in $V[G*H]$, $((\bb{T} * \dot{\bb{R}}) \times (\bb{T} * \dot{\bb{R}}))$ is $\kappa^*$-distributive, so, since $\theta^+ \leq \kappa^*$, we can apply Lemma \ref{preservationlemma} to conclude that $\mathrm{CP}(\mathcal{D})$ holds in $V[G*H]$. 
\end{proof}

\section{Forcing covering matrices} \label{forcingsect}

In this section, we show how to force the existence of covering matrices satisfying very nice properties but failing to satisfy the reflection properties $\mathrm{CP}$ and $\mathrm{S}$. We first introduce a forcing poset designed to introduce a certain type of covering matrix with very strong coherence properties. Let $\theta < \lambda$ be regular cardinals. $\bb{P}(\theta, \lambda)$ is the forcing poset whose conditions are all $p = \langle D^p(i, \beta) \mid i < \theta, \beta \leq \gamma^p \rangle$ such that:
\begin{enumerate}
	\item{$\gamma^p < \lambda$;}
	\item{for all $i < \theta$ and $\beta \leq \gamma^p$, $D^p(i, \beta)$ is a closed (below $\beta$) subset of $\beta$;}
	\item{for all $i < j < \theta$ and $\beta \leq \gamma^p$, $D^p(i, \beta) \subseteq D^p(j, \beta)$;}
	\item{for all $\beta \leq \gamma^p$, $\beta = \bigcup_{i < \theta} D^p(i, \beta)$;}
	\item{for all $i < \theta$ and $\alpha < \beta \leq \gamma^p$, if $\alpha \in D^p(i, \beta)$, then $D^p(i, \beta) \cap \alpha = D^p(i, \alpha)$;}
	\item{for all $\beta \leq \gamma^p$, there is $i < \theta$ such that $D^p(i, \beta)$ is unbounded in $\beta$.}
\end{enumerate}

If $p, q \in \bb{P}$, then $q \leq p$ iff $q$ end-extends $p$. In what follows, fix regular cardinals $\theta < \lambda$, and let $\bb{P} = \bb{P}(\theta, \lambda)$.

\begin{proposition} \label{directedclosedprop}
	$\bb{P}$ is $\theta$-directed closed.
\end{proposition}

\begin{proof}
	First note that $\bb{P}$ is tree-like. It thus suffices to verify that $\bb{P}$ is $\theta$-closed.

	Let $\delta < \theta$ be a limit ordinal, and suppose that $\vec{p} = \langle p_\eta \mid \eta < \delta \rangle$ is a strictly decreasing sequence of conditions from $\bb{P}$, where, for each $\eta < \delta$, $p_\eta = \langle D^\eta(i, \beta) \mid i < \theta, \beta \leq \gamma^\eta \rangle$. Let $\gamma^* = \sup(\{\gamma^\eta \mid \eta < \delta\})$. We will define a condition $q = \langle D^q(i, \beta) \mid i < \theta, \beta \leq \gamma^* \rangle$ that will be a lower bound for $\vec{p}$. For $i < \theta$ and $\beta < \gamma^*$, we simply let $D^q(i, \beta) = D^\eta(i, \beta)$ for some $\eta < \delta$ such that $\beta \leq \gamma^\eta$. It remains to define $D^q(i, \gamma^*)$ for $i < \theta$.

	For all $\eta_0 < \eta_1 < \delta$, let $i(\eta_0, \eta_1) < \theta$ be least such that $\gamma^{\eta_0} \in D^{\eta_1}(i(\eta_0, \eta_1), \gamma^{\eta_1})$. Let $i^* = \sup(\{i(\eta_0, \eta_1) \mid \eta_0 < \eta_1 < \delta \})$. Since $\delta < \theta$, we have $i^* < \theta$. For $i < i^*$, let $D^q(i, \gamma^*) = \emptyset$. For $i^* \leq i < \theta$, let $D^q(i, \gamma^*) = \bigcup_{\eta < \delta} D^\eta(i, \gamma^\eta)$.

	\begin{claim} \label{coherenceclaim}
		For all $\eta < \delta$ and $i^* \leq i < \theta$, $D^q(i, \gamma^*) \cap \gamma^\eta = D^q(i, \gamma^\eta)$.
	\end{claim}

	\begin{proof}
		Fix $\eta < \delta$ and $i^* \leq i < \theta$. By construction, $D^q(i, \gamma^\eta) = D^\eta(i, \gamma^\eta)$. Also by construction, $D^\eta(i, \gamma^\eta) \subseteq D^q(i, \gamma^*)$. We thus only need to show $D^q(i, \gamma^*) \cap \gamma^\eta \subseteq D^\eta(i, \gamma^\eta)$.

		Fix $\alpha \in D^q(i, \gamma^*) \cap \gamma^\eta$. By our definition of $D^q(i, \gamma^*)$, we can find $\eta_0 < \delta$ such that $\alpha \in D^{\eta_0}(i, \gamma^{\eta_0})$. If $\eta_0 = \eta$, we are done. Suppose WLOG that $\eta_0 < \eta$ (the reverse case is similar). $i(\eta_0, \eta) \leq i^* \leq i$, so $\gamma^{\eta_0} \in D^\eta(i, \gamma^\eta)$. By requirement (5) placed on conditions in $\bb{P}$, we have $D^\eta(i, \gamma^\eta) \cap \gamma^{\eta_0} = D^\eta(i, \gamma^{\eta_0})$. Since $p_\eta \leq p_{\eta_0}$, we have $D^\eta(i, \gamma^{\eta_0}) = D^{\eta_0}(i, \gamma^{\eta_0})$. In particular, $\alpha \in D^\eta(i, \gamma^\eta)$, and we are done.
	\end{proof}

	We can now verify that $q$ as we have defined it is in fact a member of $\bb{P}$. Requirements (1), (3), (4), and (6) are immediate. To verify (2), we must show that, for all $i^* \leq i < \theta$, $D^q(i, \gamma^*)$ is closed below $\gamma^*$. Fix such an $i$, and let $X$ be a bounded subset of $D^q(i, \gamma^*)$. Find $\eta < \delta$ such that $\sup(X) < \gamma^\eta$. By Claim \ref{coherenceclaim}, $X \subseteq D^\eta(i, \gamma^\eta)$. Since $D^\eta(i, \gamma^\eta)$ is closed below $\gamma^\eta$, $\sup(X) \in D^\eta(i, \gamma^\eta)$, so $\sup(X) \in D^q(i, \gamma^*)$.

	To verify (5), we must show that, for all $i^* \leq i < \theta$ and all $\alpha < \gamma^*$, if $\alpha \in D^q(i, \gamma^*)$, then $D^q(i, \gamma^*) \cap \alpha = D^q(i, \alpha)$. Fix such an $i$ and an $\alpha \in D^q(i, \gamma^*)$. Find $\eta < \delta$ such that $\alpha < \gamma^\eta$. By Claim \ref{coherenceclaim}, $\alpha \in D^\eta(i, \gamma^\eta)$. By the fact that $p_\eta$ satisfies (5), we have $D^\eta(i, \gamma^\eta) \cap \alpha = D^\eta(i, \alpha)$. Another application of Claim \ref{coherenceclaim} yields that $D^q(i, \gamma^*) \cap \alpha = (D^q(i, \gamma^*) \cap \gamma^\eta) \cap \alpha = D^\eta(i, \gamma^\eta) \cap \alpha = D^\eta(i, \alpha)$. By the construction of $q$, we have $D^q(i, \alpha) = D^\eta(i, \alpha)$, thus finishing the proof.
\end{proof}

\begin{proposition} \label{stratclosedprop}
	$\bb{P}$ is $\lambda$-strategically closed.
\end{proposition}

\begin{proof}
  We describe a winning strategy for Player II in $G_\lambda(\bb{P})$. Suppose that $0 < \xi < \lambda$, $\xi$ is an even ordinal (this includes limits), and $\langle p_\eta \mid \eta < \xi \rangle$ is a partial play of $G_\lambda(\bb{P})$ with II thus far playing according to the winning strategy to be described here. For $\eta < \xi$, let $p_\eta = \langle D^\eta(i, \beta) \mid i < \theta, \beta \leq \gamma^\eta \rangle$. Assume we have arranged inductively that, for all even ordinals $\eta_0 < \eta_1 < \xi$, we have $\gamma^{\eta_0} \in D^{\eta_1}(0, \gamma^{\eta_1})$.

	Suppose first that $\xi = \eta + 2$ for some even ordinal $\eta < \lambda$. Let $\gamma^\xi = \gamma^{\eta + 1} + 1$. We will define $p_\xi = \langle D^\xi(i, \beta) \mid i < \theta, \beta \leq \gamma^\xi \rangle$. For $i < \theta$ and $\beta < \gamma^\xi$, let $D^\xi(i, \beta) = D^{\eta + 1}(i, \beta)$. Let $i^* < \theta$ be least such that $\gamma^\eta \in D^{\eta + 1}(i, \gamma^{\eta + 1})$. For all $i < i^*$, let $D^\xi(i, \beta) = D^\eta(i, \gamma^\eta) \cup \{\gamma^\eta\}$. For all $i^* \leq i < \theta$, let $D^\xi(i, \beta) = D^{\eta + 1}(i, \gamma^{\eta + 1}) \cup \{\gamma^{\eta + 1}\}$. It is easily verified that $p_\xi \leq p_{\eta + 1}$ and satisfies our additional inductive assumption.

	Finally, suppose that $\xi$ is a limit ordinal. Let $\gamma^\xi = \sup(\{\gamma^\eta \mid \eta < \xi \})$. By the definition of the strategy in the successor case, we know that $\gamma^\eta < \gamma^\xi$ for all $\eta < \xi$. We will define $p_\xi = \langle D^\xi(i, \beta) \mid i < \theta, \beta \leq \gamma^\xi \rangle$. For $i < \theta$ and $\beta < \gamma^\xi$, let $D^\xi(i, \beta) = D^\eta(i, \beta)$ for some $\eta < \xi$ such that $\beta \leq \gamma^\eta$. For $i < \theta$, let $D^\xi(i, \gamma^\xi) = \bigcup \{D^\eta(i, \gamma^\eta) \mid \eta < \xi, \eta$ even$\}.$ The verification that $p_\xi \in \bb{P}$ and is a lower bound for $\langle p_\eta \mid \eta < \xi \rangle$ is similar to that in Proposition \ref{directedclosedprop}, using the fact, proved as in Claim \ref{coherenceclaim}, that, for all $i < \theta$ and all even $\eta < \xi$, $D^\xi(i, \gamma^\xi) \cap \gamma^\eta = D^\xi(i, \gamma^\eta)$.
\end{proof}

Thus, forcing with $\bb{P}$ preserves all cardinalities and cofinalities $\leq \lambda$. If $G$ is $\bb{P}$-generic over $V$, then $\mathcal{D} = \bigcup G$ is a closed, uniform, transitive $\theta$-covering matrix for $\lambda$ satisfying the following strong coherence condition:
\[
\mbox{for all }\alpha < \beta < \lambda \mbox{ and all } i < \theta, \mbox{ if } \alpha \in D(i, \beta), \mbox{ then } D(i, \beta) \cap \alpha = D(i, \alpha).
\]
We now show that this generically-added covering matrix $\mathcal{D}$ fails to satisfy $\mathrm{CP}(\mathcal{D})$ and $\mathrm{S}(\mathcal{D})$.

\begin{theorem}
	Let $G$ be $\bb{P}$-generic over $V$, and let $\mathcal{D} = \bigcup G$. Then, in $V[G]$, $\mathrm{CP}(\mathcal{D})$ and $\mathrm{S}(\mathcal{D})$ both fail. 
\end{theorem}

\begin{proof}
	Suppose for sake of contradiction that $\mathrm{CP}(\mathcal{D})$ holds in $V[G]$, as witnessed by $T \subseteq \lambda$. 

	\begin{claim} \label{coverClaim}
		For all $\beta < \lambda$, there is $i_\beta < \theta$ such that $T \cap \beta \subseteq D(i_\beta, \beta)$.
	\end{claim}

	\begin{proof}
		Fix $\beta < \lambda$, and suppose for sake of contradiction that there is no such $i_\beta$. Then, for all $i < \theta$, there is $\alpha_i \in T \cap \beta$ such that $\alpha_i \not\in D(i, \beta)$. Let $X = \{\alpha_i \mid i < \theta\}$. $X \in [T]^{\leq \theta}$, so there are $\gamma < \lambda$ and $j < \theta$ such that $X \subseteq D(j, \gamma)$. By increasing $j$ if necessary, we may also assume that $\beta \in D(j, \gamma)$. But then $X \subseteq D(j, \gamma) \cap \beta = D(j, \beta)$. In particular, $\alpha_j \in D(j, \beta)$, which is a contradiction.
	\end{proof}

	Since $D(i, \beta)$ is closed below $\beta$ for all $i < \theta$ and $\beta < \lambda$, Claim \ref{coverClaim} remains true if we replace $T$ by its closure, so we may in fact assume that $T$ is club in $\lambda$.

	Fix $i^* < \theta$ such that $i_\beta = i^*$ for unboundedly many $\beta < \lambda$

	\begin{claim}
		For all $\alpha \in T$, $T \cap \alpha \subseteq D(i^*, \alpha)$.
	\end{claim}

	\begin{proof}
		Fix $\alpha \in T$, and find $\beta$ such that $\alpha < \beta < \lambda$ and $i_\beta = i^*$. Then $T \cap \beta \subseteq D(i^*, \beta)$. In particular, $\alpha \in D(i^*, \beta)$, so $D(i^*, \beta) \cap \alpha = D(i^*, \alpha)$. But then $T \cap \alpha \subseteq D(i^*, \alpha)$, as desired.
	\end{proof}

	Let $\dot{T}$ be a $\bb{P}$-name for $T$ and $\dot{\mathcal{D}} = \langle \dot{D}(i, \beta) \mid i < \theta, \beta < \lambda \rangle$ be the canonical name for $\mathcal{D}$, and find $p \in G$ forcing the following:
	\begin{itemize}
		\item{$\dot{T}$ is club in $\lambda$;}
		\item{for all $\alpha \in \dot{T}$, $\dot{T} \cap \alpha \subseteq \dot{D}(i^*, \alpha)$.}
	\end{itemize}
	Back in $V$, construct $\langle p_n \mid n < \omega \rangle$, a decreasing sequence from $\bb{P}$, and an increasing sequence $\langle \eta_n \mid n < \omega \rangle$ of ordinals below $\lambda$ such that, letting $p_n = \langle D^n(i, \beta) \mid i < \theta, \beta \leq \gamma^n \rangle$, the following hold:
	\begin{itemize}
		\item{$p_0 = p$;}
		\item{for all $n < \omega$, $p_{n+1} \Vdash ``\eta_n \in \dot{T}"$;}
		\item{for all $n < \omega$, $\gamma^n < \eta_n < \gamma^{n+1}$;}
		\item{for all $n < \omega$, $\gamma^n \in D^{n+1}(0, \gamma^{n+1})$.}
	\end{itemize}
	The construction is straightforward using the winning strategy for II described in Proposition \ref{stratclosedprop}. Let $\gamma^* = \sup(\{\gamma^n \mid n < \omega \}) = \sup(\{\eta_n \mid n < \omega \})$. Define $q = \langle D^q(i, \beta) \mid i < \theta, \beta \leq \gamma^* \rangle$ as follows. For $i < \theta$ and $\beta < \gamma^*$, let $D^q(i, \beta) = D^n(i, \beta)$ for some $n < \omega$ such that $\beta \leq \gamma^n$. For $i \leq i^*$, let $D^q(i, \gamma^*) = \emptyset$. For $i^* < i < \theta$, let $D^q(i, \gamma^*) = \bigcup_{n < \omega} D^n(i, \gamma^n)$. 

	It is easily verified as before that $q \in \bb{P}$ and $q$ is a lower bound for $\langle p_n \mid n < \omega \rangle$. Thus, since $p \Vdash ``\dot{T}$ is club,$"$ we have $q \Vdash ``\gamma^* \in \dot{T}."$ Therefore, $q \Vdash ``\dot{T} \cap \gamma^* \subseteq \dot{D}(i^*, \gamma^*)"$. However, $q \Vdash ``\{\eta_n \mid n < \omega\} \subseteq \dot{T}$ and $\dot{D}(i^*, \gamma^*) = \emptyset",$ which is a contradiction. Thus, in $V[G]$, $\mathrm{CP}(\mathcal{D})$ fails. Since $\mathcal{D}$ is transitive, Lemma \ref{implicationlemma} implies that $\mathrm{S}(\mathcal{D})$ fails as well.  
\end{proof}

\section{Normal covering matrices} \label{normalsect}

In this section, we investigate normal covering matrices. We start by noting that $\mathrm{CP}$ places 
limits on the existence of normal covering matrices. We then show that, if $\kappa$ is a regular 
uncountable cardinal and $\square_\kappa$ holds, then there is a normal $\omega$-covering matrix for 
$\kappa^+$. We end by showing that the situation is different for wider covering matrices, as, if 
$\theta < \kappa$ and $\theta$ is regular and uncountable, the existence of a normal $\theta$-covering 
matrix for $\kappa^+$ implies $\kappa^\theta > \kappa$. 

If $\kappa$ is singular and $\cf(\kappa) = \theta$, then there is always a normal $\theta$-covering matrix 
for $\kappa^+$. In fact, we have the following lemma, a proof of which can be found in \cite{sharonviale}.

\begin{lemma}
  Suppose $\kappa$ is a singular cardinal and $\cf(\kappa) = \theta$. Then there is a transitive, uniform, 
  closed $\theta$-covering matrix $\mathcal{D}$ for $\kappa^+$ with $\beta_{\mathcal{D}} = \kappa$.
\end{lemma}

We will therefore be interested in this section mainly in $\theta$-covering matrices for $\lambda$ in the cases 
in which $\lambda$ is not the successor of a cardinal of cofinality $\theta$.

The following two lemmas follow from results in \cite{vialecovering} and \cite{sharonviale}. We provide proofs for completeness.

\begin{lemma} \label{largecovering}
	Suppose $\theta < \lambda$ are regular cardinals and $\mathcal{D} = \langle D(i, \beta) \mid i < \theta, \beta < \lambda \rangle$ is a $\theta$-covering matrix for $\lambda$. Suppose that $T \subseteq \lambda$ is unbounded and $\mathcal{D}$ covers $[T]^\theta$. Suppose also that $\kappa < \lambda$ is a cardinal such that either $\kappa \in [\theta, \theta^{+\theta})$ or $\kappa^\theta < \lambda$. Then $\mathcal{D}$ covers $[T]^\kappa$.
\end{lemma}

\begin{proof}
	Fix $X \in [T]^\kappa$. Assume first that $\kappa \in [\theta, \theta^{+\theta})$. We proceed by induction on such $\kappa$. The case $\kappa = \theta$ is true by hypothesis, so let $\kappa > \theta$ and suppose we have proven the Lemma for all $\mu \in [\theta, \kappa)$. Note that $\cf(\kappa) \not= \theta$. Write $X$ as an increasing union $X = \bigcup_{\alpha < \cf(\kappa)} X_\alpha$ with $|X_\alpha| < \kappa$ for all $\alpha < \cf(\kappa)$. For each $\alpha < \cf(\kappa)$, use the inductive hypothesis to find $\beta_\alpha < \lambda$ and $i_\alpha < \theta$ such that $X_\alpha \subseteq D(i_\alpha, \beta_\alpha)$. Let $\beta = \sup(\{\beta_\alpha \mid \alpha < \cf(\kappa)\})$. For each $\alpha < \cf(\kappa)$, let $j_\alpha < \theta$ be such that $X_\alpha \subseteq D(j_\alpha, \beta)$. Since $\cf(\kappa) \not= \theta$, there is $j \leq \theta$ such that, for unboundedly many $\alpha < \cf(\kappa)$, $j_\alpha \leq j$. Then $X \subseteq D(j, \beta)$.

	Next, suppose $\kappa^\theta < \lambda$. For each $Y \in [X]^\theta$, find $\beta_Y < \lambda$ and $i_Y < \theta$ such that $Y \subseteq D(i_Y, \beta_Y)$. Let $\beta = \sup(\{\beta_Y \mid Y \in [X]^\theta\})$. We claim that there is $j < \theta$ such that $X \subseteq D(j, \beta)$. Suppose not. Then, for all $j < \theta$, there is $\alpha_j \in X \setminus D(j, \beta)$. Let $Y = \{\alpha_j \mid j < \theta\}$. Then $Y \in [X]^\theta$ and, since $\beta \geq \beta_Y$, there is $j$ such that $Y \subseteq D(j, \beta)$. In particular, $\alpha_j \in D(j, \beta)$, which is a contradiction.
\end{proof}

\begin{lemma} \label{nonormal1}
	Suppose $\theta < \lambda$ are regular cardinals and either $\lambda \in (\theta^+, \theta^{+\theta})$ or $\lambda$ is $\theta$-inaccessible. Suppose moreover that $\mathrm{CP}(\lambda, \theta)$ holds. Then there are no normal, locally downward coherent $\theta$-covering matrices for $\lambda$. In particular, if additionally $2^\theta < \lambda$, there are no normal $\theta$-covering matrices for $\lambda$.
\end{lemma}

\begin{proof}
	Suppose for sake of contradiction that $\mathcal{D} = \langle D(i, \beta) \mid i < \theta, \beta < \lambda \rangle$ is a locally downward coherent $\theta$-covering matrix for $\lambda$ and $\beta_{\mathcal{D}} < \lambda$. Since $\mathrm{CP}(\mathcal{D})$ holds, there is an unbounded $T \subseteq \lambda$ such that $[T]^\theta$ is covered by $\mathcal{D}$. By Lemma \ref{largecovering}, either hypothesis implies that $[T]^{<\lambda}$ is covered by $\mathcal{D}$. Find $X \in [T]^{<\lambda}$ such that $\mathrm{otp}(X) > \beta_{\mathcal{D}}$. Then there are $i < \theta$ and $\beta < \lambda$ such that $X \subseteq D(i, \beta)$, contradicting $\mathrm{otp}(D(i, \beta)) < \beta_{\mathcal{D}}$.
\end{proof}

\begin{remark}
	In Lemma \ref{largecovering} and Lemma \ref{nonormal1}, the hypotheses ``$\kappa^\theta < \lambda$" and ``$\lambda$ is $\theta$-inaccessible" can be relaxed to ``$\cf([\kappa]^{\leq \theta}, \subseteq) < \lambda$" and ``$\cf([\kappa]^{\leq \theta}, \subseteq) <\lambda$ for all $\kappa < \lambda$," respectively. In the case in which $\theta$ is uncountable, Lemma \ref{nonormal1} is improved by Theorem \ref{nonormal2} below.
\end{remark}

\begin{corollary}
	Suppose PFA holds and $\lambda > \omega_2$ is a regular cardinal that is not the successor of a cardinal of 
  countable cofinality. Then there are no normal $\omega$-covering matrices for $\lambda$.
\end{corollary}

\begin{proof}
  Since PFA implies both SCH and $2^\omega = \omega_2$, we have that $\lambda$ is $\omega$-inaccessible and $2^\omega < \lambda$. 
  Therefore, by Lemma \ref{nonormal1}, there are no normal $\omega$-covering matrices for $\lambda$.
\end{proof}

In particular, if $\kappa \geq \omega_2$ and $\cf(\kappa) > \omega$, then PFA implies that there are no normal $\omega$-covering matrices 
for $\kappa^+$. As we will see in Theorem \ref{square_normal_thm}, this contrasts with the situation under $\square_\kappa$. The following lemma will be useful for our constructions. 
For a proof, see \cite{lambiehanson}.

\begin{lemma} \label{otplemma}
	Let $\kappa$ be an infinite, regular cardinal, and let $n < \omega$. Let $\delta < \kappa$, and suppose $\{X_\alpha \mid \alpha < \delta \}$ is such that each $X_\alpha$ is a set of ordinals and $\mathrm{otp}(X_\alpha) < \kappa^n$. If $X = \bigcup_{\alpha < \delta} X_\alpha$, then $\mathrm{otp}(X) < \kappa^n$.
\end{lemma}

\begin{theorem} \label{square_normal_thm}
	Suppose $\kappa$ is a regular uncountable cardinal and $\square_\kappa$ holds. Then there is a normal, uniform, downward coherent $\omega$-covering matrix for $\kappa^+$.
\end{theorem}

\begin{proof}
	Let $\vec{C} = \langle C_\alpha \mid \alpha < \kappa^+ \rangle$ be a $\square_\kappa$-sequence. Let $S_0 = \{\alpha < \kappa^+ \mid \omega \leq \cf(\alpha) < \kappa\}$, and let $S_1 = \{\alpha < \kappa^+ \mid \cf(\alpha) = \kappa\}$.

	We will define $\mathcal{D} = \langle D(i, \beta) \mid i < \omega, \beta < \kappa^+ \rangle$, a normal, uniform, downward coherent $\omega$-covering matrix for $\kappa^+$, by recursion on $\beta$. $\mathcal{D}$ will also satisfy the following stronger requirements:
	\begin{enumerate}
		\item{For all $i < \omega$ and $\beta < \kappa^+$, $\mathrm{otp}(D(i, \beta)) < \kappa^{i+1}$. In particular, we will have $\beta_{\mathcal{D}} = \kappa^\omega$.}
		\item{For all $\alpha < \beta < \kappa^+$ and all $i < \omega$, if $\alpha \in C'_\beta$ and $\beta \in S_0$, then $\alpha \in D(i, \beta)$ and $D(i, \beta) \cap \alpha = D(i, \alpha)$.}
		\item{For all $\alpha < \beta < \kappa^+$ and all $i < \omega$, if $\alpha \in C'_\beta$ and $\beta \in S_1$, then $\alpha \in D(i+1, \beta)$ and $D(i+1, \beta) \cap \alpha = D(i, \alpha)$.}
	\end{enumerate}

	We first deal with successor ordinals. Suppose $\beta < \kappa^+$ and $\langle D(i, \alpha) \mid i < \omega, \alpha \leq \beta \rangle$ has been defined. For all $i < \omega$, let $D(i, \beta + 1) = \{\beta\} \cup D(i, \beta)$. It is easily verified that this satisfies all of the requirements.

	Next, suppose $\beta < \kappa^+$ is a limit ordinal and $\langle D(i, \alpha) \mid i < \omega, \alpha < \beta \rangle$ has been defined. There are three cases.

	\textbf{Case 1:} $\cf(\beta) = \omega$ and $C'_\beta$ is bounded below $\beta$. 

	Let $\langle \alpha_n \mid n < \omega \rangle$ be an increasing sequence of ordinals, cofinal in $\beta$ such that, if $C'_\beta \not= \emptyset$, then $\alpha_0 = \max(C'_\beta)$. For all $i < \omega$, let $D(i, \beta) = \{\alpha_n \mid n < \omega\} \cup D(i, \alpha_0) \cup \bigcup_{n < \omega} (D(i, \alpha_{n+1}) \setminus \alpha_n)$.

	This definition easily satisfies the first two requirements in the definition of a covering matrix and the requirement of uniformity. We first verify that it satisfies condition (1) listed above. For every $i < \omega$, $D(i, \beta)$ is the union of countably many sets, each with order type less than $\kappa^{i+1}$, by the induction hypothesis. Thus, Lemma \ref{otplemma} implies that $\mathrm{otp}(D(i, \beta)) < \kappa^{n+1}$. We next verify condition (2) listed above. Suppose $\alpha \in C'_\beta$. Then $\alpha \leq \alpha_0$. By construction, for all $i < \omega$, $D(i, \beta) \cap \alpha_0 = D(i, \alpha_0)$ and, by the inductive hypothesis, since $\cf(\alpha_0) < \kappa$ and $\alpha \in C'_{\alpha_0}$, we have $D(i, \alpha_0) \cap \alpha = D(i, \alpha)$. Thus, $D(i, \beta) \cap \alpha = D(i, \alpha)$.

	Next, we verify that we have satisfied condition (3) in the definition of a covering matrix. We show by induction on $n < \omega$ that, if $\alpha \leq \alpha_n$ and $i < \omega$, there is $j < \omega$ such that $D(i, \alpha) \subseteq D(j, \beta)$. If $\alpha \leq \alpha_0$, then there is $j < \omega$ such that $D(i, \alpha) \subseteq D(j, \alpha_0)$, and $D(j, \alpha_0) \subseteq D(j, \beta)$, so $D(i, \alpha) \subseteq D(j, \beta)$. For the induction step from $n$ to $n+1$, fix $\alpha \in (\alpha_n, \alpha_{n+1}]$ and $i < \omega$. First, find $i_0 < \omega$ such that $D(i, \alpha) \subseteq D(i_0, \alpha_{n+1})$. Next, by downward coherence, find $i_1 < \omega$ such that $D(i, \alpha) \cap \alpha_n \subseteq D(i_1, \alpha_n)$. Finally, by the induction hypothesis, find $j$ with $i_0 \leq j < \omega$ such that $D(i_1, \alpha_n) \subseteq D(j, \beta)$. Then $D(i, \alpha) \cap \alpha_n \subseteq D(i_1, \alpha_n) \subseteq D(j, \beta)$, and $D(i, \alpha) \setminus \alpha_n \subseteq D(i_0, \alpha_{n+1}) \setminus \alpha_n \subseteq D(j, \beta)$, so $D(i, \alpha) \subseteq D(j, \beta)$.

	Finally, we verify downward coherence. We again proceed by induction on $n < \omega$ and show that, for all $\alpha \leq \alpha_n$ and all $i < \omega$, there is $j < \omega$ such that $D(i, \beta) \cap \alpha \subseteq D(j, \alpha)$. First, fix $\alpha \leq \alpha_0$ and $i < \omega$. By construction, $D(i, \beta) \cap \alpha_0 = D(i, \alpha_0)$. Find $j < \omega$ such that $D(i, \alpha_0) \cap \alpha \subseteq D(j, \alpha)$. Then $D(i, \beta) \cap \alpha \subseteq D(j, \alpha)$. For the induction step from $n$ to $n+1$, fix $\alpha \in (\alpha_n, \alpha_{n+1}]$ and $i < \omega$. First, find $i_0 < \omega$ such that $\alpha_n \in D(i_0, \alpha)$ and $D(i, \alpha_{n+1}) \cap \alpha \subseteq D(i_0, \alpha)$. Next, find $i_1 < \omega$ such that $D(i, \beta) \cap \alpha_n \subseteq D(i_1, \alpha_n)$. Finally, find $j$ with $i_0 \leq j < \omega$ such that $D(i_1, \alpha_n) \subseteq D(j, \alpha)$. Then $D(i, \beta) \cap \alpha_n \subseteq D(i_1, \alpha_n) \subseteq D(j, \alpha)$, and $D(i, \beta) \cap [\alpha_n, \alpha) \subseteq \{\alpha_n\} \cup (D(i, \alpha_{n+1}) \cap \alpha) \subseteq D(j, \alpha)$, so $D(i, \beta) \cap \alpha \subseteq D(j, \alpha)$.

	\textbf{Case 2:} $\cf(\beta) < \kappa$ and $C'_\beta$ is unbounded in $\beta$.

	For all $i < \omega$, let $D(i, \beta) = \bigcup_{\alpha \in C'_\beta} D(i, \alpha)$. Notice that, by our induction hypotheses, this ensures that, for all $i < \omega$, $C'_\beta \subseteq D(i, \beta)$ and, for all $\alpha \in C'_\beta$, $D(i, \beta) \cap \alpha = D(i, \alpha)$. All of the requirements are either trivial or verified as in Case 1.

	\textbf{Case 3:} $\cf(\beta) = \kappa$.

	Let $D(0, \beta) = \emptyset$ and, for all $i < \omega$, $D(i + 1, \beta) = \bigcup_{\alpha \in C'_\beta} D(i, \alpha)$. All of the requirements are trivial or verified as before with the possible exception of the order type requirement. To verify that requirement, note that, for all $\alpha \in C'_\beta$ and all $i < \omega$, $D(i+1, \beta) \cap \alpha = D(i, \alpha)$. Thus, by our induction hypothesis, $D(i+1, \beta)$ is a set of ordinals, all of whose initial segments have order type less than $\kappa^{i+1}$. Therefore, $\mathrm{otp}(D(i+1, \beta)) \leq \kappa^{i+1} < \kappa^{i+2}$, as required.  
\end{proof}

In particular, if $\kappa$ is a regular, uncountable cardinal, then it is consistent that $\kappa^\omega = \kappa$ and there is a normal $\omega$-covering matrix for $\kappa^+$. Contrast this with the following result.

\begin{theorem} \label{nonormal2}
	Suppose $\theta < \kappa$ and $\theta$ is a regular, uncountable cardinal. Suppose moreover that there is normal $\theta$-covering matrix for $\kappa^+$. Then $\kappa^\theta > \kappa$.
\end{theorem}

\begin{proof}
	Suppose $\mathcal{D} = \langle D(i, \beta) \mid i < \theta, \beta < \kappa^+ \rangle$ is a normal $\theta$-covering matrix for $\kappa^+$. Define $\mathcal{E} = \langle E(i, \beta) \mid i < \theta, \beta < \kappa^+ \rangle$ as follows.
	\begin{itemize}
		\item{$E(0, \beta) = \emptyset$ for all $\beta < \kappa^+$.}
		\item{$E(i+1, \beta) = D(i, \beta)$ for all $\beta < \kappa^+$ and $i < \theta$.}
		\item{$E(j, \beta) = \bigcup_{i<j} E(i, \beta) = \bigcup_{i<j} D(i, \beta)$ for all $\beta < \kappa^+$ and limit $j < \theta$.}
	\end{itemize}
	It is easily checked that $\mathcal{E}$ is a $\theta$-covering matrix for $\kappa^+$ and $\beta_{\mathcal{E}} = \beta_{\mathcal{D}}$. 

	For all $\beta < \kappa^+$, define a function $f_\beta : \theta \rightarrow \beta_\mathcal{E}$ by letting $f_\beta(i) = \mathrm{otp}(E(i, \beta))$.

	\begin{claim}
		Let $\alpha < \beta < \kappa^+$. There is a club $C$ in $\theta$ such that, for all $j \in C$, $f_\alpha(j) < f_\beta(j)$.
	\end{claim}

	\begin{proof}
		Let $i^* < \theta$ be least such that $\alpha \in E(i^*, \beta)$. Define a function $g : \theta \rightarrow \theta$ by letting, for all $i < \theta$, $g(i)$ be the least $j < \theta$ such that $E(i, \alpha) \subseteq E(j, \beta)$. Let $C_0 = \{j \in \mathrm{lim}(\theta) \mid g``j \subseteq j \}$, and let $C = C_0 \setminus i^*$. We claim that $C$ is as desired. To see this, fix $j \in C$. As $g``j \subseteq j$, we have that, for all $i < j$, $E(i, \alpha) \subseteq E(j, \beta)$. Since $E(j, \alpha) = \bigcup_{i<j} E(i, \alpha)$, we have $E(j, \alpha) \subseteq E(j, \beta)$. Since $i^* < j$, we have $\alpha \in E(j, \beta)$. Thus, $\mathrm{otp}(E(j, \beta)) \geq \mathrm{otp}(E(j, \alpha)) + 1$, so $f_\alpha(j) < f_\beta(j)$.
	\end{proof}

	In particular, $f_\alpha \not= f_\beta$ for all $\alpha < \beta < \kappa^+$, so ${^\theta} \beta_{\mathcal{E}} \geq \kappa^+$. Since $\beta_\mathcal{E} < \kappa^+$, this implies $\kappa^\theta \geq \kappa^+$.
\end{proof}

We do not in fact know if the hypothesis of Theorem \ref{nonormal2} is consistent if $\cf(\kappa) \neq \theta$ and ask the following question.

\begin{question}
  Is it consistent that there are regular uncountable cardinals $\theta < \lambda$ such that $\lambda$ is not the successor of a cardinal of cofinality $\theta$ 
  and there is a normal $\theta$-covering matrix for $\lambda$?
\end{question}

\section{Good points} \label{goodpointsect}

In this section, 
we introduce the notion of a good point for a covering matrix in analogy with the notion of good points in PCF-theoretic scales. We develop the basic 
theory of good points in covering matrices and use this to prove some non-existence results, focusing particularly on $\theta$-covering matrices 
for $\lambda$ in which $\theta$ and $\lambda$ are relatively close together. We are motivated in particular by the following result of Sharon and 
Viale \cite{sharonviale}, which they use to prove that MM implies the failure of the Chang's conjecture variant $(\aleph_{\omega+1}, \aleph_\omega) 
\twoheadrightarrow (\aleph_2, \aleph_1)$.

\begin{theorem}
  Suppose $(\aleph_{\omega + 1}, \aleph_\omega) \twoheadrightarrow (\aleph_2, \aleph_1)$. Then there is a transitive, uniform, normal 
  $\omega$-covering matrix for $\omega_2$.
\end{theorem}

We first need some definitions and results from Shelah's PCF theory.

\begin{definition}
	Let $\kappa$ be an infinite cardinal, and let $I$ be an ideal on $\kappa$. If $f,g \in {^\kappa}\mathrm{On}$, then $f <_I g$ if $\{i < \kappa \mid g(i) \leq f(i)\} \in I$. $=_I$ and $\leq_I$ are defined in analogous ways.

	If $D$ is a filter on $\kappa$, then $<_D$ has the same meaning as $<_I$, where $I$ is the dual ideal to $D$.
\end{definition}

\begin{remark}
	The ordering $<^*$ on ${^\kappa}\mathrm{On}$ is the same as $<_I$, where $I$ is the ideal of bounded subsets of $\kappa$.
\end{remark}

\begin{definition}
	Suppose $\kappa$ is an infinite cardinal, $I$ is an ideal on $\kappa$, and $\vec{f} = \langle f_\beta \mid \beta < \delta \rangle$ is a $<_I$-increasing sequence of functions from ${^\kappa}\mathrm{On}$. $g \in {^\kappa}\mathrm{On}$ is an \emph{exact upper bound (eub)} for $\vec{f}$ if:
	\begin{enumerate}
		\item{For all $\beta < \delta$, $f_\beta <_I g$.}
		\item{For all $h \in {^\kappa}\mathrm{On}$, if $h <_I g$, then there is $\beta < \delta$ such that $h <_I f_\beta$.}
	\end{enumerate}
\end{definition}

\begin{definition}
	Suppose $\kappa$ is an infinite cardinal and $I$ is an ideal on $\kappa$. If $\vec{f} = \langle f_\alpha \mid \alpha < \delta \rangle$ and $\vec{g} = \langle g_\beta \mid \beta < \eta \rangle$ are two $<_I$-increasing sequences of functions from ${^\kappa}\mathrm{On}$, then $\vec{f}$ and $\vec{g}$ are \emph{cofinally interleaved modulo $I$} if, for every $\alpha < \delta$, there is $\beta < \eta$ such that $f_\alpha <_I g_\beta$ and, for every $\beta < \eta$, there is $\alpha < \delta$ such that $g_\beta <_I f_\alpha$.
\end{definition}

\begin{remark}
	Let $\kappa$ be an infinite cardinal, and let $I$ be an ideal on $\kappa$. Suppose $\vec{f}$ and $\vec{g}$ are $<_I$-increasing sequences of functions from ${^\kappa}\mathrm{On}$.
	\begin{enumerate}
		\item{If $h_0$ and $h_1$ are both eubs for $\vec{f}$, then $h_0 =_I h_1$.}
		\item{If $\vec{f}$ and $\vec{g}$ are cofinally interleaved modulo $I$ and $h$ is an eub for $\vec{f}$, then $h$ is also an eub for $\vec{g}$.}
	\end{enumerate}
\end{remark}

Shelah's Trichotomy Theorem is one of the foundational results in PCF Theory.

\begin{theorem} (Trichotomy)
	Suppose $\kappa$ is an infinite cardinal, $I$ is an ideal on $\kappa$, $\kappa^+ < \lambda = \cf(\lambda)$, and $\vec{f} = \langle f_\beta \mid \beta < \lambda \rangle$ is a $<_I$-increasing sequence of functions from ${^\kappa}\mathrm{On}$. Then one of the following holds:
	\begin{enumerate}
		\item{(Good) $\vec{f}$ has an eub, $g$, such that, for all $i < \kappa$, $\cf(g(i)) > \kappa$.}
		\item{(Bad) There is an ultrafilter $U$ on $\kappa$ extending the dual filter to $I$ and a sequence $\langle S_i \mid i < \kappa \rangle$ such that $|S_i| \leq \kappa$ for all $i < \kappa$ and there is a subfamily of $\prod_{i<\kappa} S_i$ that is cofinally interleaved with $\vec{f}$ modulo $U$.}
		\item{(Ugly) There is a function $h \in {^\kappa}\mathrm{On}$ such that the sequence of sets $\langle \{i < \kappa \mid f_\beta(i) < h(i) \} \mid \beta < \lambda \rangle$ does not stabilize modulo $I$.}
	\end{enumerate}
\end{theorem}

The following result, which follows easily from Trichotomy, is useful.

\begin{theorem} \label{statGoodThm}
	Suppose $\kappa$ is an infinite cardinal, $I$ is an ideal on $\kappa$, $\kappa^+ < \lambda = \cf(\lambda)$ and $\vec{f} = \langle f_\beta \mid \beta < \lambda \rangle$ is a $<_I$-increasing sequence of functions from ${^\kappa}\mathrm{On}$. Suppose moreover that, for stationarily many $\beta \in S^\lambda_{>\kappa}$, $\langle f_\alpha \mid \alpha < \beta \rangle$ has an eub $g$ such that, for all $i < \kappa$, $\cf(g(i)) > \kappa$. Then $\vec{f}$ has an eub $h$ such that, for all $i < \kappa$, $\cf(h(i)) > \kappa$.
\end{theorem}

We now introduce the notion of a good point for a covering matrix.

\begin{definition}
	Let $\theta, \lambda$ be regular cardinals such that $\theta^+ < \lambda$, and let $\mathcal{D} = \langle D(i,\beta) \mid i < \theta, \beta < \lambda \rangle$ be a $\theta$-covering matrix for $\lambda$. If $\beta \in S^{\lambda}_{>\theta}$, then $\beta$ is a \emph{good point for $\mathcal{D}$} if there is an unbounded $A\subseteq \beta$ such that, for every $\gamma < \beta$, there is $\alpha < \beta$ and $i < \theta$ such that $A\cap \gamma \subseteq D(i,\alpha)$. $A_{\mathcal{D}}$ denotes the set of $\beta \in S^\lambda_{>\theta}$ such that $\beta$ is good for $\mathcal{D}$.
\end{definition}

\begin{remark}
  Our definition is motivated by the PCF-theoretic notion of good points in $<^*$-increasing sequences of functions. In particular, if $\mathcal{D}$ in the above definition is a transitive covering matrix, then, as we will see 
  in the proof of Theorem \ref{normalA}, the sequence of 
  functions $\vec{f} = \langle f_\beta \mid \beta < \lambda \rangle$ from ${^\theta}\beta_{\mathcal{D}}$ defined by letting $f_\beta(i) = \mathrm{otp}(D(i, \beta))$ is $<^*$-increasing, 
  and a good point $\beta$ for $\mathcal{D}$ is also a good point for $\vec{f}$ in the PCF-theoretic sense, i.e. there is an eub $g$ for $\langle f_\alpha \mid \alpha < \beta \rangle$ 
  such that $\cf(g(i)) = \cf(\beta)$ for all $i < \theta$. As later results in this section will show, good points for covering matrices are also of interest in cases 
  in which the covering matrix is not necessarily transitive. Just as the analysis of PCF-theoretic good points places limits on the behavior of $<^*$-increasing 
  sequences of functions, our analysis of good points in covering matrices will allow us to prove some non-existence results about covering matrices.
\end{remark}

\begin{theorem} \label{normalA}
	Suppose $\theta$ is a regular cardinal, $1 \leq \eta < \theta$, and $\lambda = \theta^{+\eta+1}$. Let $\mathcal{D} = \langle D(i,\beta) \mid i < \theta, \beta < \lambda \rangle$ be a $\theta$-covering matrix for $\lambda$. If $\mathcal{D}$ is transitive and $A_{\mathcal{D}}$ is stationary, then $\mathcal{D}$ is not normal.
\end{theorem}

\begin{proof}
	Suppose for sake of contradiction that $\mathcal{D}$ is transitive and normal and $A_{\mathcal{D}}$ is stationary. For each $\beta < \lambda$, define $f_\beta \in {^\theta}\beta_\mathcal{D}$ by letting $f_\beta(i) = \mathrm{otp}(D(i,\beta))$. The fact that $\mathcal{D}$ is transitive means that, for all $\alpha < \beta < \lambda$ and $i < \theta$, if $\alpha \in D(i,\beta)$, then $f_\alpha(i) < f_\beta(i)$. Thus, the sequence $\vec{f} = \langle f_\beta \mid \beta < \lambda \rangle$ is $<^*$-increasing.

	Suppose that $\beta < \lambda$ is good for $\mathcal{D}$ as witnessed by an unbounded $A \subseteq \beta$. We may assume that $\mathrm{otp}(A) = \cf(\beta)$. We start by finding an unbounded $A' \subseteq A$ such that, for every $\alpha \in A'$, there is $i < \theta$ such that $A' \cap \alpha \subseteq D(i,\alpha)$. $A'$ will be enumerated in increasing fashion by $\langle \alpha_\eta \mid \eta < \cf(\beta) \rangle$, and we define it recursively as follows:
	\begin{itemize}
		\item{$\alpha_0 = \min(A)$.}
		\item{If $\eta < \cf(\beta)$ and $\alpha_\eta$ has been defined, let $\alpha_{\eta + 1} = \min(A \setminus (\alpha_\eta + 1))$.}
		\item{If $\eta < \cf(\beta)$ is a limit ordinal and $\langle \alpha_\xi \mid \xi < \eta \rangle$ has been defined, let $\alpha_\eta$ be the least $\alpha \in A$ such that, for some $i < \theta$, $\{\alpha_\xi \mid \xi < \eta \} \subseteq D(i,\alpha)$.}
	\end{itemize}
	It is clear that $A'$ thusly defined has the desired properties. Now find $i^* < \theta$ and an unbounded $A^* \subseteq A'$ such that, for all $\alpha \in A^*$, $A^* \cap \alpha \subseteq D(i^*,\alpha)$ and note that, if $\alpha_0, \alpha_1 \in A^*$, $\alpha_0 < \alpha_1$, and $i^* \leq i < \theta$, then $f_{\alpha_0}(i) < f_{\alpha_1}(i)$. Define $g \in {^\theta}\mathrm{On}$ by letting $g(i) = \cf(\beta)$ for all $i < i^*$ and $g(i) = \sup(\{f_\alpha(i) \mid \alpha \in A^*\})$ for all $i^* \leq i < \theta$.  It is easily seen that $g$ is an eub for $\langle f_\alpha \mid \alpha < \beta \rangle$ and $\cf(g(i)) = \cf(\beta)$ for all $i < \theta$. By assumption, this is true for stationarily many $\beta \in S^\lambda_{>\theta}$, which, by Theorem \ref{statGoodThm}, implies that $\vec{f}$ has an exact upper bound $h \in {^\omega}(\beta_{\mathcal{D}}+1)$ such that $\cf(h(i)) > \theta$ for all $i < \theta$. Since $h(i) < \lambda$ for all $i < \theta$, there is $\mu \in (\theta, \theta^{+\eta+1})$ and an unbounded $B \subseteq \theta$ such that, for all $i \in B$, $\cf(h(i)) = \mu$. For each $i \in B$, let $\langle \delta^i_\xi \mid \xi < \mu \rangle$ be increasing and cofinal in $h(i)$. For $\xi < \mu$, define $h_\xi : \theta \rightarrow \beta_{\mathcal{D}}$ by letting $h_\xi(i) = 0$ for $i \not\in B$ and $h_\xi(i) = \delta^i_\xi$ for $i \in B$. Then $h_\xi < h$, so there is $\alpha_\xi < \lambda$ such that $h_\xi <^* f_{\alpha_\xi}$. Let $\alpha^* = \sup(\{\alpha_\xi \mid \xi < \mu\})$. Since $\mu < \lambda$, $\alpha^* < \lambda$. For all $\xi < \mu$, find $i_\xi < \theta$ such that, whenever $i_\xi \leq i < \theta$, $h_\xi(i) < f_{\alpha^*}(i)$. Since $\theta < \mu$, there is an $i^* < \theta$ and an unbounded $E \subseteq \mu$ such that, for all $\xi \in E$, $i_\xi = i^*$. Then, for all $i^* \leq i < \theta$ with $i \in B$, $f_{\alpha^*}(i) \geq \sup(\{h_\xi(i) \mid \xi \in E\}) = h(i)$, contradicting the fact that $h$ is an upper bound for $\vec{f}$.
\end{proof}

For our next result, we need the following club guessing theorem, due to Shelah.

\begin{theorem} \label{clubguessing}
	Suppose $\kappa$ and $\lambda$ are regular cardinals and $\kappa^+ < \lambda$. Then there is a sequence $\langle C_\alpha \mid \alpha \in S^\lambda_\kappa \rangle$ such that:
	\begin{enumerate}
		\item{For all $\alpha \in S^\lambda_\kappa$, $C_\alpha$ is club in $\alpha$.}
		\item{For every club $C \subseteq \lambda$, the set $\{\alpha \in S^\lambda_\kappa \mid C_\alpha \subseteq C\}$ is stationary.}
	\end{enumerate}
\end{theorem}

A sequence $\langle C_\alpha \mid \alpha \in S^\lambda_\kappa \rangle$ as given in Theorem \ref{clubguessing} is called a \emph{club-guessing sequence}.

\begin{theorem}
	Let $\theta < \kappa < \lambda$ be regular cardinals with $\kappa^{++} < \lambda$, and let $\mathcal{D} = \langle D(i,\beta) \mid i < \theta, \beta < \lambda \rangle$ be a uniform $\theta$-covering matrix for $\lambda$. Then $A_{\mathcal{D}} \cap S^{\lambda}_{\kappa}$ is stationary.
\end{theorem}

\begin{proof}
	Let $E$ be club in $\lambda$. We will find $\alpha \in E \cap A_{\mathcal{D}} \cap S^{\lambda}_{\kappa}$. Fix a club-guessing sequence $\langle C_\alpha \mid \alpha \in S^{\kappa^{++}}_{\kappa} \rangle$. Define an increasing, continuous sequence $\langle \beta_\eta \mid \eta < \kappa^{++} \rangle$ of elements of $E$ as follows.
	\begin{itemize}
		\item{$\beta_0 = \min(E)$.}
		\item{If $\eta < \kappa^{++}$ and $\beta_\eta$ has been defined, choose $\beta_{\eta + 1} \in E$ large enough so that, for all $\alpha \in S^{\kappa^{++}}_{\kappa}$, if there are $i < \theta$ and $\beta < \lambda$ such that $\{\beta_\xi \mid \xi \in C_\alpha \cap (\eta + 1)\} \subseteq D(i,\beta)$, then there are such an $i < \theta$ and $\beta < \lambda$ with $\beta < \beta_{\eta + 1}$.}
	\end{itemize}

	Let $\gamma = \sup(\{\beta_\eta \mid \eta < \kappa^{++}\})$. Since $\mathcal{D}$ is uniform, we can find a club $F$ in $\gamma$ such that, for some $i < \theta$, $F \subseteq D(i,\gamma)$. We can assume that $F \subseteq \{\beta_\eta \mid \eta < \kappa^{++}\}$. Let $F^* = \{\eta < \kappa^{++} \mid \beta_\eta \in F\}$. $F^*$ is club in $\kappa^{++}$, so there is some $\alpha \in S^{\kappa^{++}}_{\kappa}$ such that $C_\alpha \subseteq F^*$. We claim that $\beta_\alpha$ is good for $\mathcal{D}$, as witnessed by $\{\beta_\eta \mid \eta \in C_\alpha\}$.

	To see this, fix $\eta \in C_\alpha$. We must show that, for some $\delta < \alpha$ and $i < \theta$, $\{\beta_\xi \mid \xi \in C_\alpha \cap (\eta + 1)\} \subseteq D(i, \delta)$. Consider stage $\eta + 1$ in the construction of the sequence $\langle \beta_\xi \mid \xi < \kappa^{++} \rangle$. It is the case that there is $\beta < \lambda$ and $i < \theta$ such that $\{\beta_\xi \mid \xi \in C_\alpha \cap (\eta + 1)\} \subseteq D(i,\beta)$ (namely, $\beta = \gamma$), so, by construction, it is also the case that, for some $j < \theta$, $\{\beta_\xi \mid \xi \in C_\alpha \cap (\eta + 1)\} \subseteq D(j, \beta_{\eta + 1})$.
\end{proof}

\begin{corollary} \label{4_corollary}
	If $\theta < \lambda$ are regular cardinal and $\theta^{+4} \leq \lambda < \theta^{+\theta}$, then there are no transitive, uniform, normal $\theta$-covering matrices for $\lambda.$
\end{corollary}

\begin{remark}
  In \cite{lambiehanson}, we show that, if $\theta$ is a regular cardinal, there is consistently a transitive, uniform, normal $\theta$-covering matrix for $\theta^+$. The situations for $\theta^{++}$ and $\theta^{+3}$ are currently unclear. We also note that Corollary \ref{4_corollary} is not in conflict with Theorem \ref{square_normal_thm}, as the covering matrix constructed in the proof of 
  Theorem \ref{square_normal_thm} is not transitive.
\end{remark}

We now make some additional observations about $A_{\mathcal{D}}$, starting with results that downward coherence has an effect on $A_{\mathcal{D}}$.

\begin{proposition}
	Let $\theta$ and $\lambda$ be regular cardinals with $\theta^+ < \lambda$, and let $\mathcal{D}$ be a $\theta$-covering matrix for $\lambda$. If $\mathcal{D}$ is downward coherent, then $A_{\mathcal{D}} = S^\lambda_{>\theta}$.
\end{proposition}

\begin{proof}
	For all $\beta \in S^{\lambda}_{>\theta}$, $\beta \in A_{\mathcal{D}}$ as witnessed by $D(i,\beta)$, where $i$ is large enough so that $D(i,\beta)$ is unbounded in $\beta$.
\end{proof}

The proof of the following Lemma is almost identical to that of Lemma \ref{largecovering} and is thus omitted.

\begin{lemma} \label{coherenceStep}
	Let $\theta < \lambda$ be regular cardinals, and suppose $\mathcal{D}$ is a locally downward coherent $\theta$-covering matrix for $\lambda$. Suppose $\kappa < \lambda$ is a cardinal such that either $\kappa \in [\theta, \theta^{+\theta})$ or $\kappa^\theta < \lambda$. Then, for any $Y \in [\lambda]^\kappa$, there is $\gamma_Y$ such that, for every $i < \theta$ and $\beta < \lambda$, there is $j < \theta$ such that $D(i, \beta) \cap Y \subseteq D(j, \gamma_Y)$.
\end{lemma}

\begin{theorem} \label{coherenceA}
	Suppose $\theta$ and $\lambda$ are regular cardinals and either $\lambda \in (\theta^+, \theta^{+\theta})$ or $\lambda$ is $\theta$-inaccessible. Let $\mathcal{D}$ be a locally downward coherent $\theta$-covering matrix for $\lambda$. Then $S^\lambda_{>\theta} \setminus A_{\mathcal{D}}$ is non-stationary.
\end{theorem}

\begin{proof}
	Under either assumption, Lemma \ref{coherenceStep} implies that, for all $\delta < \lambda$, there is $\gamma_\delta < \lambda$ such that, for all $i < \theta$ and $\beta < \lambda$, there is $j < \theta$ such that $D(i, \beta) \cap \delta \subseteq D(j, \gamma_\delta)$. Suppose for sake of contradiction that $S := S^\lambda_{>\theta} \setminus A_{\mathcal{D}}$ is stationary. For each $\beta \in S$, let $i_\beta < \theta$ be such that $D(i_\beta, \beta)$ is unbounded in $\beta$. Since $\beta \not\in A_{\mathcal{D}}$, it must be the case that there is $\delta_\beta < \beta$ such that, for all $\alpha < \beta$ and all $j < \theta$, $D(i_\beta, \beta) \cap \delta_\beta \not\subseteq D(j, \alpha)$. Find a stationary $T \subseteq S$ and a fixed $\delta < \lambda$ such that, for all $\beta \in T$, $\delta_\beta = \delta$. Fix $\beta \in T \setminus (\gamma_\delta + 1)$. Then there is $j < \theta$ such that $D(i_\beta, \beta) \cap \delta \subseteq D(j, \gamma_\delta)$, contrary to the assumption that $\delta_\beta = \delta$.
\end{proof}

\begin{corollary}
	Suppose $\theta$ and $\lambda$ are regular cardinals and $\lambda \in (\theta^+, \theta^{+\theta})$. Suppose $\mathcal{D}$ is a transitive, locally downward coherent $\theta$-covering matrix for $\lambda$. Then $\mathcal{D}$ is not normal.
\end{corollary}

\begin{proof}
	Combine Theorems \ref{normalA} and \ref{coherenceA}.
\end{proof}

$\mathrm{CP}$ also has an impact on $A_{\mathcal{D}}$.

\begin{proposition}
	Let $\theta < \lambda$ be regular cardinals. Suppose $\mathcal{D}$ is a $\theta$-covering matrix for $\lambda$ and $\mathrm{CP}(\mathcal{D})$ holds. Then, for every successor ordinal $\eta < \theta$ such that $\theta^{+\eta} < \lambda$, $A_{\mathcal{D}} \cap S^\lambda_{\theta^{+\eta}}$ is stationary.
\end{proposition}

\begin{proof}
	Let $T \subseteq \lambda$ witness $\mathrm{CP}(\mathcal{D})$. Fix $E$ club in $\lambda$ and $\eta < \theta$ such that $\theta^{+\eta} < \lambda$. We will find $\beta \in E \cap A_{\mathcal{D}} \cap S^\lambda_{\theta^{+\eta}}$. To do this, we will build an increasing, continuous sequence of ordinals $\langle \beta_\xi \mid \xi < \theta^{+\eta} \rangle$ such that:
	\begin{itemize}
		\item{$\beta_0 = \min(T)$.}
		\item{For all $\xi < \theta^{+\eta}$, $\beta_{\xi+1} \in T$ and $\{\beta_{\zeta+1} \mid \zeta < \xi\} \subseteq D(i,\beta_{\xi+1})$ for some $i < \theta$.}
		\item{For all limit $\xi < \theta^{+\eta}$, $\beta_\xi \in E$.}
	\end{itemize}
	The construction is straightforward. If $\xi < \theta^{+\eta}$ and $\beta_\xi$ has been defined, let $\alpha_\xi = \min(E \setminus \beta_\xi)$ and find $\beta_{\xi+1} \in T \setminus (\alpha_\xi + 1)$ such that, for some $i < \theta$, $\{\beta_{\zeta+1} \mid \zeta < \xi\} \subseteq D(i,\beta_{\xi+1})$. This can be done by Lemma \ref{largecovering}. Now let $\beta = \sup(\{\beta_\xi \mid \xi < \theta^{+\eta}\})$. $\beta \in E \cap S^\lambda_{\theta^{+\eta}}$, and $\{\beta_{\xi+1} \mid \xi < \theta^{+\eta}\}$ witnesses that $\beta \in A_{\mathcal{D}}$.
\end{proof}

If $\theta$ and $\lambda$ are close, we can do even better.

\begin{proposition}
	Let $\theta$ be a regular cardinal, and let $\lambda = \theta^{+\eta}$ for some successor ordinal $\eta < \theta$. Suppose $\mathcal{D}$ is a $\theta$-covering matrix for $\lambda$ and $\mathrm{CP}(\mathcal{D})$ holds. Then $S^\lambda_{>\theta} \setminus A_{\mathcal{D}}$ is non-stationary.
\end{proposition}

\begin{proof}
	Let $T \subseteq \lambda$ witness $\mathrm{CP}(\mathcal{D})$. By Lemma \ref{largecovering}, for all $\alpha < \lambda$, there are $i < \theta$ and $\beta < \lambda$ such that $T \cap \alpha \subseteq D(i, \beta)$. Define a function $f:\lambda \rightarrow \lambda$ by letting $f(\alpha)$ be the least $\beta < \lambda$ such that, for some $i < \theta$, $T \cap \alpha \subseteq D(i, \beta)$. Let $C$ be the set of closure points of $f$. $C$ is a club in $\lambda$ and, if $\beta \in C$, $\cf(\beta) > \theta$, and $\beta$ is a limit point of $T$, then $T \cap \beta$ witnesses that $\beta \in A_{\mathcal{D}}$.
\end{proof}

\section{Stationary reflection and club-increasing sequences} \label{clubincreasesect}

In this section, we study certain increasing sequences of functions that arise naturally from transitive, uniform covering matrices. We show that stationary reflection hypotheses place strong limits on the behavior of these sequences. In what follows, if $\kappa$ is an infinite cardinal, $f,g \in {^\kappa}\mathrm{On}$, and $i < \kappa$, then $f <_i g$ means that $f(j) < g(j)$ for all $i \leq j < \kappa$. We first make the following definition.

\begin{definition}
	Let $\kappa$ be an infinite cardinal and $\delta$ an ordinal. $\vec{f} = \langle f_\beta \mid \beta < \delta \rangle$ is a \emph{club-increasing sequence of functions from ${^\kappa}\mathrm{On}$} if the following hold:
	\begin{enumerate}
		\item{For all $\beta < \delta$, $f_\beta$ is a non-decreasing function in ${^\kappa}\mathrm{On}$.}
		\item{For all $\alpha < \beta < \delta$, $f_\alpha <^* f_\beta$.}
		\item{For all $\beta < \delta$ such that $\mathrm{cf}(\beta) > \kappa$, there is a club $C_\beta \subseteq \beta$ and an $i_\beta < \kappa$ such that, for all $\alpha \in C_\beta$, $f_\alpha <_{i_\beta} f_\beta$.}
	\end{enumerate}
	If $\vec{f}$ is a club-increasing sequence of functions, let $\gamma_{\vec{f}}$ denote the least ordinal $\gamma$ such that, for all $\beta < \delta$, $f_\beta \in {^\kappa}\gamma$.
\end{definition}

\begin{remark}
	Suppose that $\theta < \lambda$ are infinite, regular cardinals and $\mathcal{D} = \langle D(i, \beta) \mid i < \theta, \beta < \lambda \rangle$ is a transitive, uniform $\theta$-covering matrix. For all $\beta < \lambda$, define $f_\beta \in {^\theta}\mathrm{On}$ by letting, for all $i < \theta$, $f_\beta(i) = \mathrm{otp}(D(i, \beta))$. Then $\vec{f} = \langle f_\beta \mid \beta < \lambda \rangle$ is a club-increasing sequence of functions and $\gamma_{\vec{f}} = \beta_{\mathcal{D}}$.
\end{remark}

\begin{proposition} \label{upperbdprop}
	Suppose $\theta < \lambda$ are infinite, regular cardinals and $\vec{f} = \langle f_\beta \mid \beta < \lambda \rangle$ is a club-increasing sequence of functions from ${^\theta}\mathrm{On}$. Then $\cf(\gamma_{\vec{f}})$ is either $\theta$ or $\lambda$.
\end{proposition}

\begin{proof}
	Let $\gamma = \gamma_{\vec{f}}$ and $\mu = \cf(\gamma)$. Let $\langle \gamma_\eta \mid \eta < \mu \rangle$ be increasing and cofinal in $\gamma$. Since $\gamma = \sup(\{f_\beta(i) \mid i < \theta, \beta < \lambda\})$, it is easily seen that $\mu \leq \lambda$. It is also immediate that $\gamma$ is not a successor ordinal as, if $\gamma = \gamma_0 + 1$, then there is some $\beta < \lambda$ and $i < \theta$ such that $f_\beta(j) = \gamma_0$ for all $i \leq j < \theta$. But then, since $f_\beta <^* f_{\beta + 1}$, there is $i^* < \theta$ such that $f_{\beta + 1}(i^*) \geq \gamma_0 + 1 = \gamma$. Thus, $\mu \leq \lambda$ and $\mu$ is infinite.

	Suppose for sake of contradiction that $\mu < \lambda$ and $\mu \not= \theta$. For all $\eta < \mu$, find $\beta_\eta < \lambda$ and $i_\eta < \theta$ such that $f_{\beta_\eta}(i_\eta) \geq \gamma_\eta$. Let $\beta = \sup(\{\beta_\eta + 1 \mid \eta < \mu\})$. Since $\mu < \lambda$, we have $\beta < \lambda$. For each $\eta < \mu$, find $j_\eta < \theta$ such that, for all $j_\eta \leq j < \theta$, $f_{\beta_\eta}(j) < f_\beta(j)$.

	First, suppose $\mu < \theta$. Let $j^* = \sup(\{i_\eta, j_\eta \mid \eta < \mu\})$. Since $\mu < \theta$, we have $j^* < \theta$. Also, for all $\eta < \mu$, $f_\beta(j^*) \geq \gamma_\eta$, so $f_\beta(j^*) \geq \gamma$. Contradiction.

	Next, suppose $\theta < \mu < \lambda$. Fix an unbounded $A \subseteq \mu$ and $i^*, j^* < \theta$ such that, for all $\eta \in A$, $i_\eta = i^*$ and $j_\eta = j^*$. Let $k = \max(i^*, j^*)$. Then, for all $\eta \in A$, $f_\beta(k) \geq \gamma_\eta$, so $f_\beta(k) \geq \gamma$, which is again a contradiction.
\end{proof}

\begin{corollary}
	Suppose $\theta < \lambda$ are infinite, regular cardinals and $\mathcal{D} = \langle D(i, \beta) \mid i < \theta, \beta < \lambda \rangle$ is a transitive, uniform, normal covering matrix. Then $\cf(\beta_{\mathcal{D}}) = \theta$.
\end{corollary}

\begin{theorem} \label{goodptthm}
	Suppose $\theta < \lambda$ are infinite, regular cardinals, $\theta^+ < \lambda$, $R(\lambda, \theta)$ holds, and $\vec{f} = \langle f_\alpha \mid \alpha < \lambda \rangle$ is a club-increasing sequence from ${^\theta} \mathrm{On}$. Then $\vec{f}$ falls into the `Good' case of the Trichotomy Theorem.
\end{theorem}

\begin{proof}
	Let $T \subseteq \lambda$ be a stationary set witnessing $R(\lambda, \theta)$. Suppose first that $\vec{f}$ falls into the `Bad' case, and let this be witnessed by an ultrafilter $U$ and a sequence of sets $\langle S_i \mid i < \theta \rangle$. For all $i < \theta$ and $\eta \in S_i$, let $T_{i,\eta} = \{\alpha \in T \mid f_\alpha(i) \geq \eta\}$. Define a function $h \in {^\theta}\mathrm{On}$ as follows. For each $i < \theta$, if there is $\eta \in S_i$ such that $T_{i, \eta}$ is non-stationary, let $h(i)$ be the least such $\eta$. If there is no such $\eta$, let $h(i) = \sup(S_i)$.

  Let $B = \{(i, \eta) \mid i < \theta, \eta \in S_i$, and $T_{i, \eta}$ is stationary$\}$. $|B| \leq \theta$, so, by $R(\lambda, \theta)$, there is $\beta < \lambda$ such that $T_{i, \eta}$ reflects at $\beta$ for all $(i, \eta) \in B$. Fix a club $D \subseteq \beta$ and an $i^* < \theta$ witnessing that $\vec{f}$ is club-increasing at $\beta$. For all $(i, \eta) \in B$, there is $\alpha_{i, \eta} \in D \cap T_{i, \eta}$. Thus, if, in addition, $i \geq i^*$, we have $\eta \leq f_{\alpha_{i, \eta}}(i) < f_\beta(i)$.
  
  Let $X = \{i < \theta \mid h(i) < \sup(S_i)\}$, and suppose first that $X \not\in U$. Then $\theta \setminus (X \cup i^*) \in U$ and, for all $i \in \theta \setminus (X \cup i^*)$, $f_\beta(i) \geq \sup(S_i)$, contradicting the fact that $\vec{f}$ is cofinally interleaved with a subfamily of $\prod_{i < \theta} S_i$ modulo $U$.  

  Therefore, we may assume that $X \in U$. For all $i^* \leq i < \theta$ and all $\eta \in S_i \cap h(i)$, $f_\beta(i) \geq \eta$. Find $g_\beta \in \prod_{i < \theta} S_i$ and $\gamma$ with $\beta < \gamma < \lambda$ such that $f_\beta <_U g_\beta <_U f_\gamma$. Fix $Z \subseteq (X \setminus i^*)$ such that $Z \in U$ and, for all $i \in Z$, $f_\beta(i) < g_\beta(i) < f_\gamma(i)$. Then, for all $i \in Z$, we must have $h(i) \leq g_\beta(i) < f_\gamma(i)$. Find a stationary $T^* \subseteq T \setminus (\gamma + 1)$ and a fixed $\ell \in Z$ such that, for all $\delta \in T^*$, 
  $f_\gamma(\ell) < f_\delta(\ell)$. Then $T^* \subseteq T_{\ell, h(\ell)}$, so $T_{\ell, h(\ell)}$ is stationary, contradicting the definition of $h(\ell)$.

  Next, suppose that $\vec{f}$ falls into the `Ugly' case, as witnessed by $h \in {^\theta}\mathrm{On}$. For all $i < \theta$, let $T_i = \{\alpha \in T \mid f_\alpha(i) \geq h(i)\}$. Let $A = \{i < \theta \mid T_i$ is stationary$\}$. Let $B = \theta \setminus A$.

	\begin{claim}
		Suppose $\alpha < \lambda$. Then $|\{i \in B \mid f_\alpha(i) \geq h(i)\}| < \theta$.
	\end{claim}

	\begin{proof}
		Fix $\alpha < \lambda$. For all $i \in B$, let $C_i \subseteq \lambda$ be club such that $C_i \cap T_i = \emptyset$. Let $C = \bigcap_{i \in B} C_i$. Fix $\beta \in (C \cap T) \setminus \alpha$. Then, for all $i \in B$, $f_\beta(i) < h(i)$. Since $f_\alpha \leq^* f_\beta$, the claim follows.
	\end{proof}

	By $R(\lambda, \theta)$, find $\beta < \lambda$ such that $T_i$ reflects at $\beta$ for all $i \in A$. Fix a club $D \subseteq \beta$ and an $i^* < \theta$ witnessing that $\vec{f}$ is club-increasing at $\beta$. For all $i \in A \setminus i^*$, fix $\alpha_i \in T_i \cap D$ and note that $h(i) \leq f_{\alpha_i}(i) < f_\beta(i)$. Combining this with the previous claim, we obtain that, for all $\delta$ with $\beta \leq \delta < \lambda$, $\{i < \theta \mid f_\delta(i) < h(i)\} =^* B$, which contradicts the fact that $\vec{h}$ witnesses that we are in the `Ugly' case.
\end{proof}

This result yields the following corollary. The corollary is not new, as it also follows from the results of Sharon and Viale in \cite{sharonviale}. The proof in \cite{sharonviale} relies on a detailed analysis of covering matrices while ours is more directly PCF-theoretic in nature, but the proofs are rather similar in flavor, and a comparison thereof again illustrates the connections between covering matrices and PCF-theoretic sequences of functions. Recall that, if $\mu$ is a singular cardinal and $\langle \kappa_i \mid i < \cf(\mu) \rangle$ is an increasing sequence of regular cardinals cofinal in $\mu$, a \emph{scale of length $\mu^+$ in $\prod_{i < \cf(\mu)} \kappa_i$} is a $<^*$-increasing, $<^*$-cofinal sequence $\vec{f} = \langle f_\alpha \mid \alpha < \mu^+ \rangle$ of functions in $\prod_{i < \cf(\mu)} \kappa_i$. If $\beta \in S^{\mu^+}_{>\cf(\mu)}$, then $\beta$ is a \emph{good point} for $\vec{f}$ if $\langle f_\alpha \mid \alpha < \beta \rangle$ has an eub $g$ such that $\cf(g(i)) = \cf(\beta)$ for all $i < \cf(\mu)$, and $\beta$ is a \emph{bad point} otherwise.

\begin{corollary}
	Suppose $B \subseteq \omega$ is infinite, $\vec{f} = \langle f_\alpha \mid \alpha < \aleph_{\omega + 1} \rangle$ is a scale in $\prod_{n \in B} \aleph_n$, and $R(\aleph_2, \aleph_0)$ holds. Then $\{\beta \in S^{\aleph_{\omega + 1}}_{\aleph_2} \mid \beta$ is a bad point for $\vec{f} \}$ is non-stationary.
\end{corollary}

\begin{proof}
	Using $\vec{f}$, it is routine to construct a scale $\vec{g} \in \prod_{n \in B} \aleph_n$ that is club-increasing. $\vec{g}$ and $\vec{f}$ are cofinally interleaved on a club in $\aleph_{\omega + 1}$, so it suffices to show that all points of cofinality $\aleph_2$ are good for $\vec{g}$. To this end, let $\beta \in S^{\aleph_{\omega + 1}}_{\aleph_2}$, let $\langle \alpha_\eta \mid \eta < \omega_2 \rangle$ be increasing, continuous, and cofinal in $\beta$, and consider $\vec{g}^\beta = \langle g_{\alpha_\eta} \mid \eta < \omega_2 \rangle$. $\vec{g}^\beta$ is club-increasing, so, by Theorem \ref{goodptthm}, $\vec{g}^\beta$ has an exact upper bound $h$ such that $\cf(h(n)) > \omega$ for all $n \in B$. A simple argument, similar to the proof of Proposition \ref{upperbdprop}, yields that $\cf(h(n)) = \omega_2$ for all but finitely many $n \in B$. Thus, $\beta$ is a good point for $\vec{g}$.
\end{proof}

\bibliography{covering_2}
\bibliographystyle{amsplain}
\end{document}